\documentclass[12pt,a4]{amsart}
\usepackage{amsfonts}
\usepackage{amssymb}
\usepackage{amsmath}
\usepackage{amsthm}
\usepackage{fullpage}
\theoremstyle{plain}
\newtheorem{thm}{Theorem}[section]
\newtheorem{lem}[thm]{Lemma}
\theoremstyle{definition}
\newtheorem{defn}[thm]{Definition}
\newtheorem{ex}[thm]{Example}
\newtheorem{rem}[thm]{Remark}

\numberwithin{equation}{section}
\newcommand{\R}{\mathbb{R}}

\begin{document}

\title[Perturbed Hammerstein integral equations]{Solutions of perturbed Hammerstein integral equations with applications}
\author[F. Cianciaruso]{Filomena Cianciaruso}
\address{Filomena Cianciaruso, Dipartimento di Matematica ed Informatica, Universit\`{a}
della Calabria, 87036 Arcavacata di Rende, Cosenza, Italy}
\email{cianciaruso@unical.it}
\author[G. Infante]{Gennaro Infante}
\address{Gennaro Infante, Dipartimento di Matematica ed Informatica, Universit\`{a} della
Calabria, 87036 Arcavacata di Rende, Cosenza, Italy}
\email{gennaro.infante@unical.it}
\author[P. Pietramala]{Paolamaria Pietramala}%
\address{Paolamaria Pietramala, Dipartimento di Matematica ed Informatica, Universit\`{a}
della Calabria, 87036 Arcavacata di Rende, Cosenza, Italy}
\email{pietramala@unical.it}
\subjclass[2010]{Primary 45G15, secondary 34B10, 34B18, 47H30}
\keywords{Fixed point index, cone, nontrivial solution, nonlinear boundary conditions.}
\begin{abstract}
By means of topological methods, we provide new results on the existence, non-existence, localization and multiplicity of nontrivial solutions for systems of perturbed Hammerstein integral equations. In order to illustrate our theoretical results, we study some problems that occur in applied mathematics, namely models of chemical reactors, beams and thermostats. We also apply our theory in order to prove the existence of nontrivial radial solutions of systems of elliptic boundary value problems subject to nonlocal, nonlinear boundary conditions.
\end{abstract}
\maketitle
\section{Introduction}
Problems with \emph{nonlinear} boundary conditions often occur in applied mathematics. For example, the fourth-order differential equation
\begin{equation}\label{intro-eq1}
u^{(4)}(t)=f(t, u(t)),\ t \in (0,1),%
\end{equation}
subject to the nonlinear boundary conditions (BCs)
\begin{equation}\label{intro-bc1}
u(0)= u'(0)=u^{\prime \prime}(1)=0,\  u^{\prime \prime \prime}(1)=h(u(1))
\end{equation}
models the stationary states of the deflections of an elastic beam
of length $1$. The BCs~\eqref{intro-bc1} describe that the left
end of the beam is clamped and the right end is free to move with
a vanishing bending moment and a shearing force that reacts (in a
possibly nonlinear manner) according to the displacement
registered in the right end. Various methods were used to deal
with the existence of solutions of the boundary value problem
(BVP) ~\eqref{intro-eq1}-\eqref{intro-bc1}, for example variational
methods in \cite{cate, song}, iterative methods in \cite{amp,
lizhai, mada} and topological methods in \cite{cantilever}.

One possibility is to rewrite this BVP as a \emph{perturbed} Hammerstein integral equation, that is
\begin{equation}\label{cic}
u(t)={\gamma}(t)h(u(1)) +\int_0^1 k(t,s)f(s,u(s))\,ds.
\end{equation}
This kind of perturbed integral equation  has been investigated in the past by a number of authors, we refer the reader to the manuscripts~\cite{amp, Cabada1, ac-gi-at-fpt, dfdorjp, Goodrich3, Goodrich5, Goodrich6, Goodrich7, gi-caa, kejde, paola, ya1, ya2} and references therein.

When seeking the existence of \emph{positive} solutions of the perturbed integral equation~\eqref{cic}, typically one assumes either a \emph{global} restriction on the growth of the nonlinearity $h$, say for example
\begin{equation}\label{globes}
\alpha_1 x \leq h(x) \leq \alpha_2 x,\ \text{for every}\ x\geq 0,
\end{equation}
where $0\leq \alpha_1 \leq \alpha_2$, as in~\cite{gi-caa, cantilever, gipp-nonlin,
gipp-mmas, kejde, kongwang, wang}, or an \emph{asymptotic} condition, as in~\cite{dunn-wang, Goodrich2,
Goodrich3, Goodrich5, good, Goodrich6, good1, Goodrich7}, or a kind of \emph{mixture} of the two, as in~\cite{ya1,ya2}.

Our idea is to utilise a kind of \emph{local} estimate on the
growth of the nonlinearity $h$, that can be seen as a weakening of
the global assumption~\eqref{globes}. This approach is useful
under two points of view: it allows to handle a wider class of
nonlinearities with respect to the assumption~\eqref{globes} and
is convenient in order to prove multiplicity results, henceforth
improving and complementing the above
results.

We stress that we can deal with \emph{nonlocal} BCs; for example we can replace the BCs~\eqref{intro-bc1} with
\begin{equation*}
u(0)= u'(0)=u^{\prime \prime}(1)=0,\  u^{\prime \prime \prime}(1)=h(u(\eta)),
\end{equation*}
where $\eta \in (0,1)$. This models a feedback mechanism where the
shearing force in the right end of the beam reacts to the
displacement registered in a point $\eta$.  As far as we know the
study of nonlocal BCs, in the context of ODEs, can be traced back
to Picone~\cite{Picone}, who considered multi-point BCs. For an
introduction to nonlocal problems we refer the reader to the
reviews~\cite{Conti, rma, sotiris, Stik,Whyburn} and the
papers~\cite{kttmna, ktejde, jw-gi-jlms}.

In Section~\ref{secham} we discuss the existence of solutions of
the more general equation
\begin{equation}\label{eqtham-intro}
u(t)={\gamma}(t)H[u] +\int_0^1 k(t,s)g(s)f(s,u(s))\,ds,
\end{equation}
where $H$ is a suitable compact functional in the space of continuous functions. We investigate the
existence of \emph{strictly positive}, \emph{non-negative} and
\emph{nontrivial} solutions of~\eqref{eqtham-intro}, depending on
the \emph{sign properties} of the kernel~$k$. This kind of
equation is fairly general and can be applied to a variety of
problems. As an \emph{example} we apply our results in the case of
three mathematical models, widely studied in literature, namely a
chemical reactor, a cantilever beam and a thermostat model. We also present \emph{non-existence} results
for~\eqref{eqtham-intro}. In order to illustrate our approach to
the reader, in Section~\ref{secham} we restrict our attention to
the case of \emph{one} compact perturbation of the Hammerstein
integral equation.

In Section~\ref{systems} we further develop the methodology of the previous Section and we deal with the case of systems of two perturbed Hammerstein equations. Here we focus, for brevity, to the problem of the existence of multiple, \emph{nontrivial} solutions of the system
\begin{gather}
\begin{aligned}\label{syst-intro}
u(t)=\sum_{j=1,2}\gamma_{1j}(t)H_{1j}[u,v]+\int_{0}^{1}k_1(t,s)g_1(s)f_1(s,u(s),v(s))\,ds, \\
v(t)= \sum_{j=1,2}\gamma_{2j}(t)H_{2j}[u,v]+\int_{0}^{1}k_2(t,s)g_2(s)f_2(s,u(s),v(s))\,ds,%
\end{aligned}
\end{gather}
where $H_{ij}$ are compact functionals. Some non-existence results for~\eqref{syst-intro} are also presented. Our approach allows us to deal with a wide class of systems of differential equations subject to \emph{nonlinear} nonlocal BCs. As an \emph{example} we illustrate the applicability of the theoretical results of Section~\ref{systems} by discussing the existence of nonzero radial solutions of the system of nonlinear elliptic equations subject to nonlocal, nonlinear BCs
\begin{gather}
\begin{align*}
\Delta u + \tilde{g}_1(|x|) f_1(u,v)=0,\ |&x|\in [R_1,R_0], \\
\Delta v + \tilde{g}_2(|x|) f_2(u,v)=0,\ |&x|\in [R_1,R_0],\\
\frac{\partial u}{\partial r}\Bigr\rvert_{\partial
B_{R_0}}=H_{11}[u,v]\ \text{and}\
(u(R_1 x)-\beta_1  &u(R_\eta x))\Big|_{x\in\partial B_{1}}=H_{12}[u,v],\\
\frac{\partial v}{\partial r}\Bigr\rvert_{\partial
B_{R_0}}=H_{21}[u,v] \ \text{and}\
\bigl(v(R_1 x)-\beta_2 & \frac{\partial v}{\partial r}(R_\xi x)\bigr)\Big|_{x\in\partial B_{1}}=H_{22}[u,v],%
\end{align*}
\end{gather}
where $x\in \mathbb{R}^n $, $\beta_1,\,\beta_2<0$,
$0<R_1<R_0<+\infty$, $R_\eta, R_\xi \in (R_1,R_0)$.

For our results we utilize the theory of fixed point index and make use of ideas from the earlier papers~\cite{df-gi-do, gifmpp-cnsns, gipp-nonlin, gi-pp-ft, gijwjiea, lan-lin-na, lanwebb, webb, jw-gi-jlms}.

\section{Existence and non-existence results for perturbed Hammerstein integral equations}\label{secham}
In this Section we  study the existence of solutions of the perturbed Hammerstein equation of the type
\begin{equation}\label{eqthamm}
u(t)={\gamma}(t)H[u] +\int_0^1 k(t,s)g(s)f(s,u(s))\,ds:=Tu(t),
\end{equation}
where $H$ is a  compact functional. We consider $T$ as a perturbation of the  operator
\begin{equation*}
Fu(t):=\int_0^1 k(t,s)g(s)f(s,u(s))\,ds,
\end{equation*}
that is
\begin{equation*}
Tu(t)={\gamma}(t)H[u]+Fu(t).
\end{equation*}

We work in suitable cones of  the space of continuous functions $C[0,1]$, endowed with the usual supremum norm
 $\|w\|:=\max\{|w(t)|,\; t\; \in [0,1]\}$. We recall that a \emph{cone} $K$ in a Banach space $X$  is a closed
convex set such that $\lambda \, x\in K$ for $x \in K$ and
$\lambda\geq 0$ and $K\cap (-K)=\{0\}$. If $\Omega$ is a open bounded subset of a cone $K$ (in the relative
topology) we denote by $\overline{\Omega}$ and $\partial \Omega$
the closure and the boundary relative to $K$. When $\Omega$ is an open
bounded subset of $X$ we write $\Omega_K=\Omega \cap K$, an open subset of
$K$.

The next Lemma summarizes some classical results regarding the fixed point index, for more details see for example the review by Amann~\cite{Amann-rev} and the book by Guo and Lakshmikantham~\cite{guolak}.
\begin{lem}
Let $\Omega$ be an open bounded set with $0\in \Omega_{K}$ and $\overline{\Omega}_{K}\ne K$. Assume that $F:\overline{\Omega}_{K}\to K$ is
a compact map such that $x\neq Fx$ for all $x\in \partial \Omega_{K}$. Then the fixed point index $i_{K}(F, \Omega_{K})$ has the following properties.
\begin{itemize}
\item[(1)] If there exists $e\in K\setminus \{0\}$ such that $x\neq Fx+\lambda e$ for all $x\in \partial \Omega_K$ and all $\lambda
>0$, then $i_{K}(F, \Omega_{K})=0$.
\item[(2)] If  $\mu x \neq Fx$ for all $x\in \partial \Omega_K$ and for every $\mu \geq 1$, then $i_{K}(F, \Omega_{K})=1$.
\item[(3)] If $i_K(F,\Omega_K)\ne0$, then $F$ has a fixed point in $\Omega_K$.
\item[(4)] Let $\Omega^{1}$ be open and bounded in $X$ with $\overline{\Omega^{1}_K}\subset \Omega_K$. If $i_{K}(F, \Omega_{K})=1$ and
$i_{K}(F, \Omega_{K}^{1})=0$, then $F$ has a fixed point in $\Omega_{K}\setminus \overline{\Omega_{K}^{1}}$. The same result holds if
$i_{K}(F, \Omega_{K})=0$ and $i_{K}(F, \Omega_{K}^{1})=1$.
\end{itemize}
\end{lem}
In what follows, with an abuse of notation, we denote by $w$ the
constant function $w(t)=w$ for all $t \in [0,1]$. We now discuss the relation between the existence of strictly positive, non-negative and nontrivial solutions of \eqref{eqthamm} and the sign properties of the kernel $k$, in the line of the papers \cite{ac-gi-at-bvp, gijwjiea}.
 \subsection{Strongly positive kernels} We begin by considering the case of kernels with a strong positivity and we make the following assumptions on the terms that occur in~\eqref{eqthamm}.
\begin{itemize}
\item  $k:[0,1] \times[0,1]\rightarrow (0,+\infty)$ is measurable and for every $\tau\in [0,1] $ we have
\begin{equation*}
\lim_{t \to \tau} |k(t,s)-k(\tau,s)|=0 \;\text{    for almost every (a.e.)  } s \in [0,1] .
\end{equation*}{}
\item There exist a function $\Phi \in L^{\infty}[0,1]$ and a constant $c_{1} \in (0,1]$ such that
$$
c_{1}\Phi(s) \leq  k(t,s)\leq \Phi(s) \text{ for } t \in [0,1] \text{ and a.e. } \, s\in [0,1].
$$
\item  $g\,\Phi \in L^1[0,1] $, $g(s) \geq 0$ for a.e. $s\in [0,1]$ and $\int_0^1 \Phi(s)g(s)\,ds >0$.{}
\item   The nonlinearity $f:[0,1]\times [0,+\infty) \to [0,+\infty)$ satisfies Carath\'{e}odory conditions, that is, $f(\cdot,u)$ is measurable for each fixed $u\in [0,+\infty)$ , $f(t,\cdot)$ is continuous for a.e. $t\in [0,1]$, and for each $r>0$, there exists $\phi_{r} \in L^{\infty}[0,1] $ such that{}
\begin{equation*}
f(t,u)\le \phi_{r}(t) \;\text{ for all } \; u\in [0,r]\;\text{ and a.e. } \; t\in [0,1] .
\end{equation*}{}
\item $\gamma \in C [0,1] $ and there exists $c_{2} \in(0,1] \;\text{such that}\; \gamma(t) \geq c_{2}\|\gamma\| \;\text{for}\; t \in [0,1]$.
\end{itemize}
Due to the hypotheses above, we are able to work in the cone
$$
K:=\{u\in C[0,1]:\ \min_{t\in [0,1]}u(t)\ge c\|u\|\},
$$
with $c=\min\{c_1,c_2\}$.
Regarding the nonlinear functional $H$, we assume that
\begin{itemize}
\item$H: K\to [0,+\infty)$ is  compact.
\end{itemize}
 If the hypotheses above hold  then $T$  maps $K$ into $K$ and  is compact.\\
For our index calculations we use the following open bounded
sets (relative to $K$)
$$
K_{\rho}:=\{u\in K: \|u\|<\rho\},\quad V_{\rho}:=\{u \in K:
\displaystyle{\min_{t\in [0,1]}}u(t)<\rho\}.
$$
Note that the sets $K_{\rho}$ and $V_{\rho}$ are nested, i.e. $K_{\rho}\subset V_{\rho}\subset K_{\rho/c}$.

Firstly, we give a  condition which implies that the index is $1$ on the set $K_{\rho}$.
\begin{lem} \label{ind11}
Assume that
\begin{enumerate}
\item[$(\mathrm{I}^{1}_{\rho})$] there exist
 $\rho>0$, a linear functional $\alpha^{\rho}[\cdot]:K\rightarrow [0,+\infty)$ given by
$$
\alpha^{\rho}[u]=\int_0^1 u(t)\,dA^{\rho}(t)
$$
such that
\begin{itemize}
\item $dA^{\rho}$ is a positive Stieltjes measure,
\item $\alpha^{\rho}[\gamma]<1$,
\item $H[u]\leq \alpha^{\rho}[u]$ for every $u\in \partial K_{
\rho}$,
\item the following inequality holds:
\begin{equation}\label{indice12}
  f^{c\rho,\rho}\Bigl(\sup_{t \in [0,1] }\Bigl\{\frac{\gamma(t)}{1-\alpha^{\rho}[\gamma]}\int_0^1\mathcal{K}^{\rho}(s)g(s)\,ds+\int_0^1k(t,s)g(s)\,ds\Bigr\} \Bigr)< 1,
\end{equation}
where
\begin{equation*}
 f^{c\rho,\rho}:=\mathrm{ess }\sup\left\{\frac{f(t,u)}{\rho},\,\,0\le t\le 1,\,c\rho\le
u \le \rho\right\}
\,\text{  and    } \,\mathcal{K}^{\rho}(s):=\int_0^1 k(t,s)\,dA^{\rho}(t).
\end{equation*}
\end{itemize}
\end{enumerate}
Then $i_{K}(T,K_{\rho})$ is $1$.
\end{lem}

\begin{proof}
We show that $\mu\, u\neq Tu$ for every $u\in \partial K_{\rho}$ and for every $\mu \geq 1$; this ensures that the index is 1 on $K_{\rho}$. In fact, if this does
not happen, there exist $\mu \geq 1$ and $u\in \partial K_{\rho}$ such that $\mu\, u(t)= Tu(t)$, for every $t \in [0,1]$.
Then we have
\begin{equation}\label{rel}
\mu\, u(t)\le \gamma(t)\alpha^{\rho}[u]+\int_{0}^{1} k(t,s)g(s)f(s,u(s))ds.
 \end{equation}
Applying $\alpha^{\rho}$ to the both sides of \eqref{rel} gives
\begin{equation*}
\mu\, \alpha^{\rho}[u]\le \alpha^{\rho}[\gamma]\alpha^{\rho}[u]+\int_0
^1\mathcal{K}^{\rho}(s)g(s)f(s,u(s))ds.
 \end{equation*}
 Thus we have
\begin{equation} \label{alpha}
\alpha^{\rho}[u]\le
\frac{1}{\mu-\alpha^{\rho}[\gamma]}\int_0
^1\mathcal{K}^{\rho}(s)g(s)f(s,u(s))ds\le\frac{1}{1-\alpha^{\rho}[\gamma]}\int_0
^1\mathcal{K}^{\rho}(s)g(s)f(s,u(s))ds.
 \end{equation}
 Using \eqref{alpha} in \eqref{rel} we obtain
$$
\mu\, u(t)\le\rho f^{c\rho,\rho}\left(\frac{\gamma(t)}{1-\alpha^{\rho}[\gamma]}\int_0^1\mathcal{K}^{\rho}(s)g(s)ds+\int_{0}^{1} k(t,s)g(s)ds\right).
$$
 Taking the supremum in $[0,1]$ gives
\begin{align*}
\mu\rho&\le
\rho f^{c\rho,\rho}\left(\sup_{t \in[0,1]}\left\{\frac{\gamma(t)}{1-\alpha^{\rho}[\gamma]}\int_0
^1\mathcal{K}^{\rho}(s)g(s)ds+\int_{0}^{1}
k(t,s)g(s)ds\right\}\right)
\end{align*}
and using the hypothesis \eqref{indice12} we can conclude that
$\mu \rho <\rho$. This contradicts the fact that $\mu \geq 1$ and proves the result.
\end{proof}

Now we give a  condition which implies that the index is $0$ on the set $V_{\rho}$.
\begin{lem} \label{indice0*}
Assume that
\begin{enumerate}
\item[$(\mathrm{I}^{0}_{\rho})$] there exist
 $\rho>0$, a linear functional $\alpha^{\rho}[\cdot]:K\rightarrow [0,+\infty)$ given by
 $$\alpha^{\rho}[u]=\int_0^1 u(t)\,dA^{\rho}(t)$$
such that
\begin{itemize}
\item $dA^{\rho}$ is a positive Stieltjes measure,
\item $\alpha^{\rho}[\gamma]<1$,
\item $H[u]\geq \alpha^{\rho}[u]$ for every $u\in \partial V_{
\rho}$,
\item the following inequality holds:
\begin{equation}\label{ind 0 2 +}
  f_{\rho,\rho/c}\Bigl(\inf_{t \in [0,1] }\Bigl\{\frac{\gamma(t)}{1-\alpha^{\rho}[\gamma]}\int_0^1\mathcal{K}^{\rho}(s)g(s)\,ds+\int_0^1 k(t,s) g(s)\,ds\Bigr\} \Bigr)> 1,
\end{equation}
where
\begin{equation*}
f_{\rho, \rho/c}:=\mathrm{ess }\inf\left\{\frac{f(t,u)}{\rho},\,
0\le t\le 1,\,\rho\le u \le \rho/c\right\}.
\end{equation*}
\end{itemize}
\end{enumerate}
Then $i_{K}(T,V_{\rho})$ is $0$.
\end{lem}

\begin{proof}
Note that the constant function $e(t)\equiv
1$  for $t\in[0,1]$ belongs to $K$. We prove that $u\not=Tu+\lambda e$ for every $u\in \partial V_{\rho}$ and for every $\lambda \geq 0$; this ensures that the index is $0$ on $V_{\rho}$.

Let $u\in \partial V_{\rho}$ and $\lambda \geq 0$ such that
$u=Tu+\lambda\,e$. Then we have, for $t\in[0,1]$,
\begin{align}\label{Eq0}
   u(t) \geq & \gamma(t)\alpha^{\rho}[u]+\int_{0}^{1}k(t,s)g(s)f(s,u(s))ds.
  \end{align}
Thus we have
\begin{equation*}
\alpha^{\rho}[u]\geq \alpha^{\rho}[\gamma]\alpha^{\rho}[u]+\int_{0}^{1}
\mathcal{K}^{\rho}(s)g(s)f(s,u(s))\,ds.
\end{equation*}
This implies
\begin{equation}\label{Eq1}
  \alpha^{\rho} [u]\geq  \frac{1}{1-\alpha^{\rho}[\gamma] }\int_{0}^{1} \mathcal{K}^{\rho}(s)g(s)f(s,u(s))\,ds.
\end{equation}
Using \eqref{Eq1} in \eqref{Eq0}  we obtain
$$
 u(t) \geq \rho f_{\rho,\rho/c}\Bigl(\frac{\gamma(t)}{1-\alpha^{\rho}[\gamma]}\int_0^1\mathcal{K}^{\rho}(s)g(s)\,ds+\int_0^1
k(t,s)g(s)\,ds\Bigr).
$$
Taking the infimum for $t\in [0,1]$ gives
\begin{equation*}
\rho \geq \rho f_{\rho,\rho/c} \Bigl(\inf_{t \in [0,1]
}\Bigl\{\frac{\gamma(t)}{1-\alpha^{\rho}[\gamma]}\int_0^1\mathcal{K}^{\rho}(s)g(s)\,ds+\int_0^1
k(t,s) g(s)\,ds\Bigr\} \Bigr).
 \end{equation*}
  Thus from \eqref{ind 0 2 +}  we have $\rho> \rho$.
This is a contradiction that proves the result.
\end{proof}
We now state  a result regarding the existence of at least one \emph{positive} solution of \eqref{eqthamm}. The proof
follows by the properties of fixed point index and is omitted. By
expanding the lists in conditions $(S_{1}),(S_{2})$, it is
possible to state multiplicity results,
see for example Theorem~\ref{mult-sys} and the paper~\cite{kljdeds}.
\begin{thm}\label{mult-equ}
The integral equation \eqref{eqthamm} has at least one positive solution
in $K$ if one of the following conditions holds.

\begin{enumerate}
\item[$(S_{1})$]  There exist $\rho _{1},\rho _{2}\in (0,+\infty )$ with $\rho
_{1}/c<\rho_{2}$ such that $(\mathrm{I}_{\rho _{1}}^{0})$ and  $(\mathrm{I}_{\rho _{2}}^{1})$ hold.

\item[$(S_{2})$] There exist $\rho _{1},\rho _{2}\in (0,+\infty )$ with $\rho
_{1}<\rho_{2}$ such that $(\mathrm{I}_{\rho _{1}}^{1})$ and $(\mathrm{I}_{\rho _{2}}^{0})$ hold.
\end{enumerate}
\end{thm}
In the following example we present an application of Theorem~\eqref{mult-equ} to a model of a chemical reactor.
\begin{ex}
We consider the second order ordinary differential equation
\begin{equation}\label{eq1}
u''(t)-\lambda u'(t) + \lambda \mu (\beta - u(t))^+ e^{u(t)}=0,\ t \in (0,1),
\end{equation}
subject to  the nonlinear BCs
\begin{equation}\label{bcA}
u'(0)=\lambda u(0),\ u'(1)=H[u],
\end{equation}
where $\beta, \lambda, \mu>0$ and $H$ is a suitable functional.

This BVP arises when modelling the steady states of an adiabatic
chemical reactor of length 1. Here $\lambda$ is the Peclet
number, $\mu$ is the Damkohler number, $\beta$ is the
dimensionless adiabatic temperature rise and $u(t)$ is the local
temperature at a point $t$ of the tube, see for example
\cite{feng, gippreactor, Maam} and references therein.

A local version of the BVP  \eqref{eq1}-\eqref{bcA}, that is
$$
u''(t)-\lambda u'(t) + \lambda \mu (\beta - u(t))^+ e^{u(t)}=0,\ u'(0)=\lambda u(0),\ u'(1)=0,
$$
has been studied, via different methods, by a number of authors,
we refer the reader to \cite{BOKR, CLR, feng, HOTI, madbouly, Maam, RCH, Saad} and references therein.

The nonlinear condition in~\eqref{bcA} can describe, for example,
a feedback control system on the reactor that adds or removes heat
according to the temperatures detected by some sensors located along the tube.

The solutions of the BVP~\eqref{eq1}-\eqref{bcA} are given
by the solutions of  the perturbed Hammestein integral equation
\begin{equation*}
u(t)=\gamma(t)H[u]+\int_{0}^{1}k(t,s)f(u(s))\,ds,
\end{equation*}
where
\begin{equation*}
\gamma(t)=\frac{1}{\lambda}e^{\lambda(t-1)},\quad
k(t,s)=\begin{cases}
e^{\lambda (t-s)},\ &s>t,\\
1,\ &s\leq t,
\end{cases}
\text{ and }
f(u)=\begin{cases}
\mu (\beta - u) e^{u},\ &u\leq \beta, \\
0,\ & u>\beta.
\end{cases}
\end{equation*}
We work in the cone
\begin{equation*}
K = \{u\in C[0,1], \; \min_{t\in [0,1]}u(t) \geq c \|u\|\},
\end{equation*}
 where the constant $c= e^{-\lambda}$, since
$$
e^{-\lambda} \leq k(t,s)\leq 1\text{ for } t \in [0,1]\times [0,1],
$$
and the conditions on $k$ and $\gamma$ are satisfied with  $\Phi(s)=1$ and $c_1=c_2=e^{-\lambda}$.

In order to illustrate the growth conditions that occur in Theorem~\ref{mult-equ}, we consider the BVP
\begin{equation}\label{ex-chem-eq}
u''(t)-\frac{1}{3} u'(t) + \frac{3}{10} \Bigl(\frac{11}{5} - u(t)\Bigr)^+ e^{u(t)}=0,\ t \in (0,1),
\end{equation}
\begin{equation}\label{ex-chem-bc}
u'(0)=\frac{1}{3} u(0),\, u'(1)=10^{-\frac{3}{2}}\sqrt{u(1/2)}.
\end{equation}

The choice $$\rho_1=\frac{71}{1000},
\,\,\rho_2=\frac{53}{25}(<\frac{11}{5}),
\,\,\alpha^{\rho_1}[u]=\frac{1}{10}u(1/2)
,\,\,\alpha^{\rho_2}[u]=10^{-\frac{5}{4}}u(1/2),$$
yields (in what follows the numbers are rounded to the third decimal place unless exact)
\begin{itemize}
\item[] $\alpha^{\rho_1}[\gamma]=0.254<1$ and $\alpha^{\rho_2}[\gamma]=0.143<1$,
\item[]  $H[u]=10^{-\frac{3}{2}}\sqrt{u(1/2)}
\geq \frac{1}{10}u(1/2)=\alpha^{\rho_1}[u]$ for $\rho_1\leq
u\leq\rho_1/c$,
\item[]  $H[u]=10^{-\frac{3}{2}}\sqrt{u(1/2)}\leq 10^{-\frac{5}{4}}u(1/2)=\alpha^{\rho_2}[u]$ for $c\rho_2\leq u\leq\rho_2$,
\item[]  $\inf\left\{f(u):\,\,u\in [\rho_1,
\rho_1/c]\right\}=2.057>\frac{71}{1000}\cdot 1.917$,
\item[]  $\sup\left\{f(u):\,\,u\in \left[c\rho_2,
\rho_2\right]\right\}=2.811<\frac{53}{25}\cdot 2.551$.
\end{itemize}
Thus the conditions $(\mathrm{I}^{0}_{\rho_1})$ of Lemma
\ref{indice0*} and $(\mathrm{I}^{1}_{\rho_2})$ of Lemma
\ref{ind11} are satisfied. By Theorem~\ref{mult-equ} it follows that
the BVP~\eqref{ex-chem-eq}-\eqref{ex-chem-bc} has a strictly positive solution $u \in {K}_{\rho_2}\setminus \overline{V}_{\rho_1}$ with the following localization property
$$
\rho_1=71/1000 \leq u(t)\leq 53/25= \rho_2,\ \text{for every}\
t\in [0,1].
$$
\end{ex}
\subsection{Non-negative Kernels}
We now consider the case of kernels that have a weaker positivity
property and therefore we make, in this Subsection, the following
assumptions on the terms that occur in~\eqref{eqthamm}.
\begin{itemize}
\item  $k:[0,1] \times [0,1]\rightarrow [0,+\infty)$ is measurable, and for every $\tau\in
[0,1] $ we have
\begin{equation*}
\lim_{t \to \tau} |k(t,s)-k(\tau,s)|=0 \;\text{    for a.e.   } s \in [0,1] .
\end{equation*}{}
\item  There exist $ [a,b]\subseteq [0,1]$, a function $\Phi \in L^{\infty}[0,1]$ and a constant $c_{1} \in (0,1]$ such that
\begin{align*}
k(t,s)\leq \Phi(s) \text{ for } &t \in [0,1] \text{ and a.e.}\,\, s\in [0,1], \\
k(t,s) \geq c_{1}\Phi(s) \text{ for } &t\in [a,b] \text{ and a.e.} \,\, s \in [0,1].
\end{align*}%
{}
\item $g\,\Phi \in L^1[0,1] $, $g(s) \geq 0$ for a.e. $s\in [0,1]$ and $\int_a^b \Phi(s)g(s)\,ds >0$.{}
\item The nonlinearity $f:[0,1]\times [0,+\infty) \to [0,+\infty)$ satisfies Carath\'{e}odory conditions, that is, $f(\cdot,u)$ is measurable for each fixed $u\in [0,+\infty)$ , $f(t,\cdot)$ is continuous for a.e. $t\in [0,1]$, and for each $r>0$, there exists $\phi_{r} \in L^{\infty}[0,1] $ such that{}
\begin{equation*}
f(t,u)\le \phi_{r}(t) \;\text{ for all } \; u\in [0,r]\;\text{ and a.e. } \; t\in [0,1] .
\end{equation*}{}
\item$\gamma:[0,1]\rightarrow [0,+\infty)$ is continuous and there exists $c_{2} \in(0,1] \;\text{such that}\; \gamma(t) \geq c_{2}\|\gamma\| \;\text{for}\; t \in [a,b]$.
\end{itemize}
Thus we work in the cone (with an abuse of notation)
\begin{equation}\label{cone2}
K:=\{u\in C[0,1]:\,u\geq 0,\,\, \min_{t\in [a,b]}u(t)\ge c\|u\|\},
\end{equation}
with $c=\min\{c_1,c_2\}$ and assume that
\begin{itemize}
\item$H: K\to [0,+\infty)$ is  compact.
\end{itemize}
Under the assumptions above, $T$ leaves the cone \eqref{cone2}
invariant and is compact.\\
The cone of \emph{non-negative} functions~\eqref{cone2} was firstly used by Krasnosel'ski\u\i{}, see~\cite{krzab}, and D.~Guo, see e.g.~\cite{guolak}.
With an abuse of notation,  we make use of the following open
bounded sets (relative to $K$)
$$
K_{\rho}:=\{u\in K: \|u\|<\rho\},\quad V_{\rho}:=\{u \in K: \displaystyle{\min_{t\in [a,b]}}u(t)<\rho\}.
$$
Note that $K_{\rho}\subset V_{\rho}\subset K_{\rho/c}$.\\
 We state the following results that correspond to Lemma~\ref{ind11} and Lemma~\ref{indice0*}. 
\begin{lem} \label{ind12nonneg}
Assume that
\begin{enumerate}
\item[$(\mathrm{I}^{1}_{\rho})$] the hyphoteses of Lemma \ref{ind11} hold with condition \eqref{indice12} replaced by
\begin{equation*}
  f^{0,\rho}\Bigl(\sup_{t \in [0,1] }\Bigl\{\frac{\gamma(t)}{1-\alpha^{\rho}[\gamma]}\int_0^1\mathcal{K}^{\rho}(s)g(s)\,ds+\int_0^1k(t,s)g(s)\,ds\Bigr\} \Bigr)< 1,
\end{equation*}
where
\begin{equation*}
 f^{0,\rho}:=\mathrm{ess }\sup\Bigl\{\frac{f(t,u)}{\rho},\,\, t\in [0,1],\,\,\,0\le
u \le \rho\Bigr\}.
\end{equation*}
\end{enumerate}
Then $i_{K}(T,K_{\rho})$ is $1$.
\end{lem}
\begin{lem} \label{ind 0 2 nonneg}
Assume that
\begin{enumerate}
\item[$(\mathrm{I}^{0}_{\rho})$] the hyphoteses of Lemma \ref{indice0*}  hold with condition \eqref{ind 0 2 +} replaced by
\begin{equation*}
  f_{\rho,\rho/c}\Bigl(\inf_{t \in [a,b] }\Bigl\{\frac{\gamma(t)}{1-\alpha^{\rho}[\gamma]}\int_a^b\mathcal{K}^{\rho}(s)g(s)\,ds+\int_a^b k(t,s) g(s)\,ds\Bigr\} \Bigr)> 1,
  \end{equation*}
  where
\begin{equation*}
f_{\rho,\rho/c}:=\mathrm{ess }\inf\Bigl\{\frac{f(t,u)}{\rho},\,\,
t\in [a,b],\,\,\,\rho\le u \le \rho/c\Bigr\}.
\end{equation*}
\end{enumerate}
Then $i_{K}(T,V_{\rho})$ is $0$.
\end{lem}
A result regarding the existence of \emph{non-negative} solutions, similar to Theorem~\ref{mult-equ}, holds in this case. We omit, for brevity, the statement of this result.

We conclude this Subsection with an application to a cantilever beam model.
\begin{ex}
Consider the fourth order differential equation
\begin{equation}\label{de}
u^{(4)}(t)=g(t)f(t,u(t)),\ t \in (0,1),
\end{equation}
subject to the nonlinear BCs
\begin{equation}\label{debc}
u(0)=u'(0)=u''(1)=0,\; u'''(1)+H[u]=0.
\end{equation}

Equation~\eqref{de} models the stationary states of the deflection
of an elastic beam. This kind of equation has been studied by
several authors under a variety of BCs that describe physical
constrains on the beam, for example being clamped, hinged or
mixed. The homogeneous BCs
$$
u(0)=u'(0)=u''(1)=u'''(1)=0
$$
model the so-called cantilever bar, that is a bar clamped on the
left end and where the right end is free to move with vanishing
bending moment and shearing force. These type of BCs have been
investigated in \cite{da-jh, li, yao-cant} and, in particular,
\cite{da-jh} provides a detailed insight on the physical
motivation for this problem. The BCs~\eqref{debc} model a
cantilever bar with a controller acting on the shearing force of
its right end, see for example \cite{cate, cantilever,  lizhai,
song} and the references therein.

We associate to the BVP~\eqref{de}-\eqref{debc}
the perturbed Hammerstein integral equation
\begin{equation*}
u(t)=\gamma(t)H[u]+ \int_{0}^{1} k(t,s)g(s)f(s,u(s))\,ds,
\end{equation*}
where (see  for example \cite{cantilever}) $\gamma$ and $k$ are
given by
 $$
 \gamma (t)=\frac{1}{6}(3t^2-t^3),\quad
 k(t,s)=\begin{cases} \frac{1}{6}(3t^2s-t^3),\ &s\geq t,\\ \frac{1}{6}(3s^2t-s^3),\ &s\leq
 t.
\end{cases}
 $$
The function $\Phi$ is given by
$$
\Phi (s)= \frac{1}{2}s^2-\frac{1}{6}s^3.
$$
This form for $\Phi$ corrects the typo ($ \Phi (s)=
\frac{1}{2}s^2-\frac{1}{2}s^3 $)  present in \cite{cantilever}.

For $[a,b]\subset (0,1]$, the choice
\begin{equation*}
c=\frac{1}{2}a^2(3-a)
\end{equation*}
enable us to utilize the cone
\begin{equation*}
K=\bigl\{u\in C[0,1]: \,u\geq 0,\,\, \min_{t \in [a,b]}u(t)\geq
c\|u\|\bigr\}.
\end{equation*}

In order to illustrate the constants that occur in our theory, we
consider the BVP
\begin{equation}\label{beam1}
u^{(4)}(t)=t^2u^4(t),\ t \in (0,1),
\end{equation}
\begin{equation}\label{beam2}
u(0)=u'(0)=u''(1)=0,\;\,\,\, u'''(1)+u^2(3/4)=0.
\end{equation}

Then with the choice
$$[a,b]=\left[\frac{1}{2},1\right],\,\,\rho_1=\frac{1}{2},\,\,\rho_2=5,\,\,
\alpha^{\rho_1}[u]=3\,u(3/4),\,\,
\alpha^{\rho_2}[u]=\frac{1}{2}\,u(3/4),$$ we obtain:
\begin{itemize}
\item[] $c={5}/{16}$, $\alpha^{\rho_1}[\gamma]=0.633<1$ and
$\alpha^{\rho_2}[\gamma]=0.105<1$,
\item[] $H[u]=u^2(3/4)\leq 3\,u(3/4)=\alpha^{\rho_1}[u]$ for $0\leq u\leq \rho_1$,
\item[] $H[u]=u^2(3/4)\geq \frac{1}{2}\,u(3/4)=\alpha^{\rho_2}[u]$ for $\rho_2\leq
u\leq\rho_2/c$,
\item[] $\sup\left\{f(t,u):\,\,(t,u)\in [0,1]\times[0,
\rho_1]\right\}=0.062<\frac{1}{2}\cdot 2.837$,
\item[] $\inf\left\{f(t,u):\,\,(t,u)\in \left[\frac{1}{2},1\right]\times\left[\rho_2,
\rho_2/c\right]\right\}=156.25>5\cdot 24.837$.
\end{itemize}
Thus the conditions $(\mathrm{I}^{1}_{\rho_1})$ of
Lemma~\ref{ind12nonneg} and $(\mathrm{I}^{0}_{\rho_2})$ of
Lemma~\ref{ind 0 2 nonneg} are satisfied. It follows that there
exists a non-negative solution $u \in {V}_{\rho_2}\setminus
\overline{K}_{\rho_1}$ of \eqref{beam1}-\eqref{beam2} such that
$$
0\leq u(t)\leq 16=\rho_2/c\text{ for }t \in [0,1] \mbox{ and
}u(t)\geq5/32= c\rho_1 \text{ for }t \in
\Bigl[\frac{1}{2},1\Bigr].
$$
\end{ex}

\subsection{Kernels that change sign} We now study the case of the kernels that are allowed to change sign. In this Subsection we make the following assumptions on the terms that occur in~\eqref{eqthamm}.
\begin{itemize}
\item  $k:[0,1] \times[0,1]\rightarrow \R$ is measurable, and for every $\tau\in [0,1] $ we have
\begin{equation*}
\lim_{t \to \tau} |k(t,s)-k(\tau,s)|=0 \;\text{    for a.e.   } s \in [0,1] .
\end{equation*}{}
\item  There exist $[a,b]\subseteq [0,1]$, a function $\Phi \in L^{\infty}[0,1]$ and a constant $c_{1} \in (0,1]$ such that
\begin{align*}
|k(t,s)|\leq \Phi(s) \text{ for } &t \in [0,1] \text{ and a.e. }\, s\in [0,1], \\
k(t,s) \geq c_{1}\Phi(s) \text{ for } &t\in [a,b] \text{ and a.e. } \, s \in [0,1].
\end{align*}
{}
\item $g\,\Phi \in L^1[0,1] $, $g(s) \geq 0$ for a.e.  $s\in [0,1]$ and $\int_a^b \Phi(s)g(s)\,ds >0$.{}
\item   The nonlinearity $f:[0,1]\times \mathbb{R} \to [0,+\infty)$ satisfies Carath\'{e}odory conditions, that is, $f(\cdot,u)$ is measurable for each fixed $u\in \R$, $f(t,\cdot)$
 is continuous for a.e. $t\in [0,1]$, and for each $r>0$ there exists $\phi_{r} \in L^{\infty}[0,1] $ such that{}
\begin{equation*}
f(t,u)\le \phi_{r}(t) \;\text{ for all } \; u\in [-r,r],\;\text{ and a.e. } \; t\in [0,1] .
\end{equation*}{}
\item $\gamma:[0,1]\rightarrow \mathbb{R}$ is continuous and there exists $c_{2} \in(0,1]$ such that $\gamma(t) \geq c_{2}\|\gamma\| $ for $ t \in [a,b]$.
\end{itemize}
Therefore we work in the cone
\begin{equation}\label{cone3}
K:=\{u\in C[0,1]:\ \min_{t\in [a,b]}u(t)\ge c\|u\|\},
\end{equation}
with $c=\min\{c_1,c_2\}$ and assume that
\begin{itemize}
\item$H: K\to [0,+\infty)$ is  a compact operator.
\end{itemize}
\begin{rem}
Note that the functions in the cone~\eqref{cone3} are positive on the
subset $[a,b]$ but are allowed to change sign in $[0,1]$. This cone is similar to the cone~\eqref{cone2} and
has been introduced by Infante and Webb in \cite{gijwjiea}.
\end{rem}
With the assumptions above the operator $T$ leaves the cone \eqref{cone3} invariant and is compact.\\
We use the open bounded sets
$$
K_{\rho}:=\{u\in K: \|u\|<\rho\},\quad V_{\rho}:=\{u \in K:
\displaystyle{\min_{t\in [a,b]}}u(t)<\rho\}.
$$
Note that the sets $K_{\rho}$ and $V_{\rho}$ are nested.

Now, we give a condition which implies that  the index is $1$ on the set $K_{\rho}$.

\begin{lem} \label{ind12+-}

Assume that
\begin{enumerate}
\item[$(\mathrm{I}^{1}_{\rho})$] the hyphoteses of Lemma \ref{ind11} hold with condition \eqref{indice12} replaced by
\begin{equation}\label{ind12change}
  f^{-\rho,\rho}\Bigl(\sup_{t \in [0,1] }\Bigl\{\frac{|\gamma(t)|}{1-\alpha^{\rho}[\gamma]}\int_0^1\mathcal{K}^{\rho}(s)g(s)\,ds+\int_0^1|k(t,s)|g(s)\,ds\Bigr\} \Bigr)< 1,
\end{equation}
where
\begin{equation*}
f^{-\rho,\rho}:=\mathrm{ess }\sup\left\{\frac{f(t,u)}{\rho},\,\,
t\in [0,1],\,\,\,-\rho\le u \le \rho\right\}.
\end{equation*}
\end{enumerate}
Then $i_{K}(T,K_{\rho})$ is $1$.
\end{lem}

\begin{proof}
If there exist $\mu\geq 1$ and $u \in \partial K_{\rho}$ such that  $\mu u(t)=Tu(t)$ for $t\in [0,1]$, then we have
\begin{equation*}
 \mu \alpha^{\rho}[u]=\alpha^{\rho}[\gamma] H[u]+\int_{0}^{1} \mathcal{K}^{\rho}(s)g(s)f(s,u(s))\,ds\leq \alpha^{\rho}[\gamma] \alpha^{\rho}[u] +\int_{0}^{1} \mathcal{K}^{\rho}(s)g(s)f(s,u(s))\,ds.
\end{equation*}
This implies
\begin{equation}\label{EqB}
  \alpha^{\rho} [u]\leq   \frac{1}{\mu-\alpha^{\rho}[\gamma] }\int_{0}^{1} \mathcal{K}^{\rho}(s)g(s)f(s,u(s))\,ds\leq  \frac{1}{1-\alpha^{\rho}[\gamma] }\int_{0}^{1} \mathcal{K}^{\rho}(s)g(s)f(s,u(s))\,ds.
\end{equation}
Since we have, for $t\in [0,1]$,
\begin{align*}
\mu| u(t)|\leq & |\gamma(t)|\alpha^{\rho} [u]+\int_{0}^{1} |k(t,s)|g(s)f(s,u(s))\,ds,
\end{align*}
using \eqref{EqB}  we obtain
\begin{align*}
\mu |u(t)|&\leq
\rho f^{-\rho,\rho} \Bigl(\frac{|\gamma(t)|}{1-\alpha^{\rho}[\gamma]}\int_0^1\mathcal{K}^{\rho}(s)g(s)\,ds+\int_0^1|k(t,s)|g(s)\,ds\Bigr).
 \end{align*}
Taking the supremum for $t\in [0,1]$, as in Lemma \ref{ind11}, from \eqref{ind12change}  we have $\mu \rho< \rho$.
This contradicts the fact that $\mu\geq 1$ and proves the result.  \hfill
\end{proof}

Now, we give a  condition which implies that the index is $0$ on the set $V_{\rho}$.
\begin{lem} \label{ind 0 2 + -}
Assume that
\begin{enumerate}
\item[$(\mathrm{I}^{0}_{\rho})$] the hyphoteses of Lemma \ref{indice0*}   hold with condition \eqref{ind 0 2 +} replaced by
\begin{equation}\label{ind 0 2 change}
  f_{\rho,\rho/c}\Bigl(\inf_{t \in [a,b] }\Bigl\{\frac{\gamma(t)}{1-\alpha^{\rho}[\gamma]}\int_a^b\mathcal{K}^{\rho}(s)g(s)\,ds+\int_a^b k(t,s) g(s)\,ds\Bigr\} \Bigr)> 1,
\end{equation}
where
\begin{equation*}
 f_{\rho,\rho/c}:=\mathrm{ess }\inf\Bigl\{\frac{f(t,u)}{\rho},\,\, t\in [a,b],\,\,\,\rho\le
u \le \rho/c\Bigr\}.
\end{equation*}
\end{enumerate}
Then $i_{K}(T,V_{\rho})$ is $0$.
\end{lem}

\begin{proof}
Note that the constant function $e(t)\equiv 1$  for $t\in[0,1]$
belongs to $K$. If there exist $\lambda\geq 0$ and $u \in
\partial V_{\rho}$ such that $u=Tu+\lambda e$, then we have for
$t\in [a,b]$
\begin{align*}
   u(t)&\geq  \gamma(t)\alpha^{\rho}[u]+\int_{a}^{b}k(t,s)g(s)f(s,u(s))\,ds.
  \end{align*}
Thus we have
\begin{equation*}
\alpha^{\rho}[u]\geq \alpha^{\rho}[\gamma]\alpha^{\rho}[u]+\int_{a}^{b}
\mathcal{K}^{\rho}(s)g(s)f(s,u(s))\,ds.
\end{equation*}
This implies
\begin{equation}\label{EqC}
  \alpha^{\rho} [u]\geq  \frac{1}{1-\alpha^{\rho}[\gamma] }\int_{a}^{b} \mathcal{K}^{\rho}(s)g(s)f(s,u(s))\,ds.
\end{equation}
Using \eqref{EqC}  we obtain
\begin{align*}
 u(t)&\geq
  \rho f_{\rho,\rho/c}\Bigl(\frac{\gamma(t)}{1-\alpha^{\rho}[\gamma]}\int_a^b\mathcal{K}^{\rho}(s)g(s)\,ds+\int_a^b
k(t,s)g(s)\,ds\Bigr).
 \end{align*}
Taking the infimum for $t\in [a,b]$ and using condition \eqref{ind 0 2 change}, we have $\rho> \rho$, a contradiction.  \hfill
\end{proof}
A result similar to Theorem \ref{mult-equ} holds in this case,
providing the existence of \emph{nontrivial} solutions.

We conclude this Subsection with an application to a thermostat model.
\begin{ex}\label{example}
We  consider the BVP
\begin{equation}\label{eqD}
-u''(t)=f(t,u(t)),\ \, t \in (0,1),
\end{equation}
with BCs
\begin{equation}\label{eqE}
u'(0)+H[u]=0,\; \beta u'(1) + u(\eta)=0,\; {\eta}\in [0,1].
\end{equation}

One motivation for studying this type of BVP is that it occurs in
some heat flow problems. For example the BVP
$$
-u''(t)=f(t,u(t)),\quad u'(0)+\alpha u(\xi)=0,\ \beta u'(1)+u(\eta)=0,\ \xi,\eta
\in[0,1],
$$
models the stady-state of a heated bar of length 1, where two controllers at $t = 0$ and $t = 1$ add or remove heat
according to the temperatures detected by two sensors at $t = \xi$
and $t = \eta$.
Thermostat problems of this kind were studied by Infante and Webb in \cite{gijwnodea}, who were motivated by
 Guidotti
and Merino \cite{Guidotti}. Thermostat models with \emph{linear} controllers have been studied, for example in \cite{fama, df-gi-jp-prse, gijwnodea, gijwems,
 nipime, jwpomona, jwwcna04, webb-therm} and with \emph{nonlinear} controllers in
\cite{gi-caa, kamkee, kamkee2, kapala, nipime, palamides}.

By a
solution of the BVP \eqref{eqD}-\eqref{eqE} we mean a
solution $u\in C[0,1]$ of the corresponding integral equation
$$
u(t)=(\beta + \eta- t )H[u]
+\int_{0}^{1}k(t,s)f(s,u(s))ds,
$$
 where
\begin{equation*}
k(t,s)=\beta +\begin{cases} \eta-s,\ &s\leq \eta,\\0,\ &s>\eta,
\end{cases}
-\begin{cases} t-s,\ &s\leq t,\\ 0,\ &s>t.
\end{cases}
\end{equation*}

We consider the case $\beta>0$ and $\beta +\eta < 1$, that leads to the case of solutions that are positive on an interval $[0,b]$
for every $b$ with $0\leq a <b<\eta+\beta$. We note that  in
$[0,1]\times  [0,1]$ the kernel $k$ is not positive for
$$
\beta+ \eta\leq t\leq1\ \text{and}\ 0\leq s \leq t-\beta.
$$
Upper and lower bounds for $|k|$ and $\gamma$ can be found
for example in \cite{gipp cmuc, gijwnodea,gijwems} as follows
$$
\Phi (s)=\|\gamma\|=\begin{cases}
\beta +\eta, & \beta+\eta\geq \frac{1}{2}, \\
1 -(\beta + \eta ), & \beta+\eta< \frac{1}{2},
\end{cases}
$$
 $\min_{t\in [a,b]}\gamma
(t)=\beta+\eta-b$,
$$
\min_{t\in [a,b]} k(t,s)=k(b,s)\geq
\begin{cases}
\beta, & b\leq \eta, \\
\beta + \eta -b, & b> \eta.
\end{cases}
$$
and we can choose
\begin{equation}\label{c-bcE}
c=\begin{cases}
\beta / (\beta+\eta), & b\leq \eta,\  \beta+\eta\geq \frac{1}{2}, \\
\beta / (1 -(\beta + \eta )), & b\leq \eta,\ \beta+\eta< \frac{1}{2}, \\
(\beta + \eta -b) / (\beta+\eta), & b> \eta,\  \beta+\eta\geq \frac{1}{2},\\
(\beta + \eta -b)/(1 -(\beta + \eta )), & b> \eta,\ \beta+\eta< \frac{1}{2}.
\end{cases}
\end{equation}

Therefore we work in the cone
\begin{equation*}
K = \{u\in C[0,1], \; \min_{t\in [a,b]}u(t) \geq c \|u\|\},
\end{equation*}
with $c$ as in \eqref{c-bcE}.

Now we illustrate the growth conditions and consider the BVP
\begin{equation}\label{therm1}
-u''(t)=t^2\,u^2(t)+2|u(t)|+\frac{1}{10},\,\,\,\,t\in (0,1),
\end{equation}
\begin{equation}\label{therm2}
u'(0)+u^2(1/5)=0,\;\,\,\,\,  u'(1) + 4\,u(1/4)=0.
\end{equation}
Then with the choice $$[a,b]=[0,1/4],
\,\,\rho_1=\frac{1}{3},\,\,\rho_2=\frac{31}{10},
\,\,\alpha^{\rho_1}[u]=\frac{1}{2}u(1/5),\,\,\alpha^{\rho_2}[u]=3u(1/5),$$
we obtain:
\begin{itemize}
\item[] $c=1/2$, $\alpha^{\rho_1}[\gamma]=0.15<1$ and
$\alpha^{\rho_2}[\gamma]=0.9<1$,
\item[]$H[u]=u^2(1/5)\leq \frac{1}{2}u(1/5)=\alpha^{\rho_1}[u]$ for $c\rho_1\leq
u\leq\rho_1$,
\item[] $H[u]=u^2(1/5)\geq 3u(1/5)=\alpha^{\rho_2}[u]$ for $\rho_2\leq
u\leq\rho_2/c$,
\item[] $\sup\left\{f(t,u):\,\,t\in [0,1],\,\,u\in \left[-\rho_1,
\rho_1\right]\right\}=0.88<\frac{1}{3}\cdot 2.704$,
\item[] $\inf\left\{f(t,u):\,\,t\in [0,1/4],\,\,u\in \left[\rho_2,
\frac{\rho_2}{c}\right]\right\}=6.3>\frac{31}{10}\cdot 1.85$.
\end{itemize}
Thus the conditions $(\mathrm{I}^{1}_{\rho_1})$ of Lemma
\ref{ind12+-} and $(\mathrm{I}^{0}_{\rho_2})$ of Lemma \ref{ind 0
2 + -} are satisfied. It follows that there exists a nontrivial
solution $u \in {V}_{\rho_2}\setminus \overline{K}_{\rho_1}$ of
BVP \eqref{therm1}-\eqref{therm2} with
$$
-\rho_2/2=-31/5 \leq u(t)\leq 31/5=\rho_2/2\ \text{for }\ t\in
[0,1] \text{ and }
 u(t)\geq 1/6=c\rho_1 \text{ for }\
t\in \Bigl[0,\frac{1}{4}\Bigr].
$$
\end{ex}

\subsection{Non-existence results for perturbed Hammerstein integral equations}
We now prove some non-existence results  for the integral equation
\eqref{eqthamm}. We focus, for brevity, on the case when both the
function $\gamma$ and the kernel $k$ are allowed to change sign.

\begin{thm}
Assume that there exists
 a linear functional $\alpha[\cdot]:K\rightarrow [0,+\infty)$ given by
 $$\alpha[u]=\int_0^1 u(t)\,dA(t)$$ such that
\begin{itemize}
\item $dA$ is a positive Stieltjes measure,
\item $\alpha[\gamma]< 1$,
\item $H[u]\leq \alpha[u]$ for every $u\in K$,
\item $f(t,u)<m_{\alpha}|u|$ for every $t\in [0,1]$ and $u \in \R\setminus \{0\}$,
where
\begin{equation*}
 \frac{1}{m_{\alpha}}:=\sup_{t \in [0,1] }\left\{\frac{|\gamma(t)|}{1-\alpha[\gamma]}\int_0^1\mathcal{K}(s)g(s)\,ds+\int_0^1|k(t,s)|g(s)\,ds\right\},
\end{equation*}
and
\begin{equation*}
\mathcal{K}(s):=\int_0^1 k(t,s)\,dA(t).
\end{equation*}
\end{itemize}
Then the equation  \eqref{eqthamm} has at most the function zero
as solution  in $K$.
\end{thm}
\begin{proof}
Suppose that  there exists $u \in K$ with $\|u\|=\nu>0$ such that  $u(t)=Tu(t)$ for $t\in [0,1]$. Then we have
\begin{equation*}
 \alpha[u]=\alpha[\gamma] H[u]+\int_{0}^{1} \mathcal{K}(s)g(s)f(s,u(s))\,ds\leq \alpha[\gamma] \alpha[u] +\int_{0}^{1} \mathcal{K}(s)g(s)f(s,u(s))\,ds.
\end{equation*}
This implies
\begin{equation}\label{EqH}
  \alpha [u] \leq  \frac{1}{1-\alpha[\gamma] }\int_{0}^{1} \mathcal{K}(s)g(s)f(s,u(s))\,ds.
\end{equation}
Using \eqref{EqH}  we obtain, for $t\in [0,1]$,
\begin{align*}
 |u(t)|& <  m_{\alpha}\nu \Bigl(\frac{|\gamma(t)|}{1-\alpha[\gamma]}\int_0^1\mathcal{K}(s)g(s)\,ds+\int_0^1|k(t,s)|g(s)\,ds \Bigr).
 \end{align*}
Taking the supremum for $t\in [0,1]$ gives $\nu< \nu$, a contradiction.  \hfill
\end{proof}

\begin{thm} \label{nonexi2}
Assume that there exists
 a linear functional $\alpha[\cdot]:K\rightarrow [0,+\infty)$ given by
 $$\alpha[u]=\int_0^1 u(t)\,dA(t)$$ such that
\begin{itemize}
\item $dA$ is a positive Stieltjes measure,
\item $\alpha[\gamma]< 1$,
\item $H[u]\geq \alpha[u]$ for every $u\in K$,
\item $f(t,u)>M_{\alpha} u$ for every $t\in [a,b]$ and $u \in (0,+\infty)$
where
\begin{equation*}
 \frac{1}{M_{\alpha}}:=\inf_{t \in [a,b] }\Bigl\{\frac{\gamma(t)}{1-\alpha[\gamma]}\int_a^b\mathcal{K}(s)g(s)\,ds+\int_a^b k(t,s) g(s)\,ds\Bigr\}.
\end{equation*}
\end{itemize}
Then the equation  \eqref{eqthamm} has at most the function zero
as solution  in $K$.
\end{thm}
\begin{proof}
Suppose that there exists  $u \in K$ with $\displaystyle
\min_{t\in [a,b]}u(t)=\theta>0$ such that $u=Tu$. Then we have, for
 $t\in [a,b]$,
\begin{align*}
   u(t)=&\gamma(t)H[u]+\int_{0}^{1} k(t,s)g(s)f(s,u(s))\,ds
   \geq  \gamma(t)\alpha[u]+\int_{a}^{b}k(t,s)g(s)f(s,u(s))\,ds.
  \end{align*}
Thus we have
\begin{equation*}
\alpha[u]\geq \alpha[\gamma] \alpha[u]+\int_{a}^{b} \mathcal{K}(s)g(s)f(s,u(s))\,ds.
\end{equation*}
This implies
\begin{equation}\label{EqI}
  \alpha[u]\geq  \frac{1}{1-\alpha[\gamma] }\int_{a}^{b} \mathcal{K}(s)g(s)f(s,u(s))\,ds.
\end{equation}
Using \eqref{EqI}  we obtain
\begin{align*}
 u(t)&>  M_{\alpha}\theta
\Bigl(\frac{\gamma(t)}{1-\alpha[\gamma]}\int_a^b\mathcal{K}(s)g(s)\,ds+\int_a^b
k(t,s) g(s)\,ds \Bigr).
 \end{align*}
Taking the infimum for $t\in [a,b]$ we have
 $\theta> \theta$, a contradiction.
\end{proof}
\begin{ex}
As in the example \ref{example}, consider the BVP
\begin{equation}\label{noexisex}
-u''(t)=\lambda(\sqrt{u(t)}+u^2(t)),\,\,\,\,t\in (0,1),\,\,\,u'(0)+u^2(1/5)=0,\;\,\,\,\,  u'(1) + 4\,u(1/4)=0.
\end{equation}
Choosing $\alpha[u]\equiv 0$, we have $M_\alpha=16$. Then, by Theorem \ref{nonexi2}, the BVP\eqref{noexisex} has no solution in $K$ for $\lambda> 2^{14/3}/3$.
\end{ex}

\section{Existence and non-existence results for systems of perturbed integral equations}\label{systems}
In this Section we develop an existence theory for multiple
nontrivial solutions of the system of perturbed Hammerstein integral
equations of the type
\begin{gather}
\begin{aligned}\label{syst}
u(t)=\sum_{j=1,2}\gamma_{1j}(t)H_{1j}[u,v]+F_1(u,v)(t), \\
v(t)= \sum_{j=1,2}\gamma_{2j}(t)H_{2j}[u,v]+F_2(u,v)(t),
\end{aligned}
\end{gather}
where
\begin{equation*}
F_i(u,v)(t):=\int_{0}^{1}k_i(t,s)g_i(s)f_i(s,u(s),v(s))\,ds
\end{equation*}
and $H_{ij}$ are compact  functionals not necessarily linear.
Systems of perturbed Hammerstein integral equations were studied in a number of papers, see for example \cite{df-gi-do, Goodrich2, Goodrich1, gifmpp-cnsns, gipp-ns, gipp-nonlin,  kang-wei, kejde, ya1} and references therein.

We state some assumptions on the terms that occur in the system~\eqref{syst}.
\begin{itemize}
\item For every $i=1,2$, $f_i: [0,1]\times \mathbb{R}\times \mathbb{R} \to
[0,+\infty)$ satisfies Carath\'{e}odory conditions, that is,
$f_i(\cdot,u,v)$ is measurable for each fixed $(u,v)\in
\mathbb{R}\times \mathbb{R}$, $f_i(t,\cdot,\cdot)$ is continuous
for a.e. $t\in [0,1]$, and for each $r>0$ there exists $\phi_{i,r}
\in L^{\infty}[0,1]$ such that{}
\begin{equation*}
f_i(t,u,v)\le \phi_{i,r}(t) \;\text{ for } \; u,v\in
[-r,r]\;\text{ and a.e.} \; t\in [0,1].
\end{equation*}%
{}
\item For every $i=1,2$, $k_i:[0,1]\times [0,1]\to \mathbb{R}$ is
measurable, and for every $\tau\in [0,1]$ we have
\begin{equation*}
\lim_{t \to \tau} |k_i(t,s)-k_i(\tau,s)|=0 \;\text{ for a.e.}\, s
\in [0,1].
\end{equation*}%
{}
\item For every $i=1,2$, there exist a subinterval $[a_i,b_i] \subseteq
[0,1] $, a function $\Phi_i \in L^{\infty}[0,1]$ and a constant
$\tilde{c}_{i} \in (0,1]$, such that
\begin{align*}
|k_i(t,s)|\leq \Phi_i(s) \text{ for } &t \in [0,1] \text{ and a.e.}\, s\in [0,1], \\
k_i(t,s) \geq \tilde{c}_{i}\Phi_i(s) \text{ for } &t\in [a_i,b_i]
\text{ and a.e.} \, s \in [0,1].
\end{align*}%
{}
\item For every $i=1,2$, $g_i\,\Phi_i \in L^1[0,1]$, $g_i \geq 0$ a.e. and $%
\int_{a_i}^{b_i} \Phi_i(s)g_i(s)\,ds >0$.
\item  For every $i,j=1,2$, $\gamma_{ij} \in C[0,1]$  and
there exists $c_{ij} \in(0,1]$ such that
\begin{equation*}
\gamma _{ij}(t)\geq c_{ij}\| \gamma _{ij}\|\;\text{for every%
}\;t\in \lbrack a_{i},b_{i}].
\end{equation*}%
\end{itemize}
Due to the hypotheses above, we work in the space $C[0,1]\times C[0,1]$ endowed with the norm $\|
(u,v)\| :=\max \{\| u\| ,\| v\| \}$ (with an abuse of notation). We use the cone $K$ in
$C[0,1]\times C[0,1]$ defined by
\begin{equation*}
\begin{array}{c}
K:=\{(u,v)\in \tilde{K_{1}}\times \tilde{K_{2}}\},
\end{array}
\end{equation*}
where for $i=1,2$
\begin{equation*}
\tilde{K_{i}}:=\{w\in C[0,1]:\,\,\min_{t\in [a_{i},b_{i}]}w(t)\geq
{c_{i}}\| w\|\},
\end{equation*}%
whit $c_{i}=\min \{\tilde{c}_{i},c_{i1},c_{i2}\}$. We assume that
\begin{itemize}
\item  For every $i,j=1,2$, $H_{ij}: K\to [0,+\infty)$ is a compact functional.
\end{itemize}

For a \emph{nontrivial} solution of the system \eqref{syst} we mean a solution
$(u,v)\in K$ of \eqref{syst} such that $\|(u,v)\|>0$.

Under our assumptions, it is possible to show that the integral operator
\begin{gather*}
\begin{aligned}
T(u,v)(t):=
\left(
\begin{array}{c}
\displaystyle\sum_{j=1,2}\gamma_{1j}(t)H_{1j}[u,v]+F_1(u,v)(t) \\
\displaystyle\sum_{j=1,2}\gamma_{2j}(t)H_{2j}[u,v]+F_2(u,v)(t)%
\end{array}
\right)
:=
\left(
\begin{array}{c}
T_1(u,v)(t) \\
T_2(u,v)(t)%
\end{array}
\right)
\end{aligned}
\end{gather*}
leaves the cone $K$ invariant and is compact, see for example Lemma 1 of~\cite{gifmpp-cnsns}.
We use the following (relative) open bounded sets in $K$:
\begin{equation*}
K_{\rho_1,\rho_2} = \{ (u,v) \in K : \|u\|< \rho_1\ \text{and}\ \|v\|< \rho_2\}
\end{equation*}
and
\begin{equation*}
V_{\rho_1,\rho_2} =\{(u,v) \in K: \min_{t\in [a_1,b_1]}u(t)<\rho_1\ \text{and}\ \min_{t\in
[a_2,b_2]}v(t)<\rho_2\}.
\end{equation*}
 The set $V_{\rho,\rho}$ (in the context of systems) was introduced by the authors in~\cite{gipp-ns} and is equal to the set
called $\Omega^{\rho /c}$ in~\cite{df-gi-do}, an extension to the
case of systems of a set given by Lan \cite{lan}.

For our index calculations we make use of the following Lemma, similar to Lemma $5$ of \cite{df-gi-do}. We use  different radii, in the spirit of the papers~\cite{chzh2, gipp-nodea}. This choice allows more freedom in the growth of the nonlinearities. The proof of the Lemma is similar to the corresponding one in~\cite{df-gi-do} and is omitted.

\begin{lem} 
The sets $K_{\rho_1,\rho_2}$ and $V_{\rho_1,\rho_2}$ have the
following properties:
\begin{enumerate}
\item $K_{\rho_1,\rho_2}\subset V_{\rho_1,\rho_2}\subset K_{\rho_1/c_1,\rho_2/c_2}$.
\item $(w_1,w_2) \in \partial V_{\rho_1,\rho_2}$ \; iff \; $(w_1,w_2)\in K$, $\displaystyle
\min_{t\in [a_i,b_i]} w_i(t)= \rho_i$ for some $i\in \{1,2\}$ and $\displaystyle
\min_{t\in [a_j,b_j]}w_j(t)
\le \rho_j$ for $j\neq i$.
\item If $(w_1,w_2) \in \partial V_{\rho_1,\rho_2}$, then for some $i\in\{1,2\}$ $\rho_i \le w_i(t) \le \rho_i/c_i$
 for each $t \in [a_i,b_i]$ and for $j\neq i$ we have $0 \leq w_j(t) \leq \rho_j/c_j$ for each $t\in [a_j,b_j]$.
 \item $(w_1,w_2) \in \partial K_{\rho_1,\rho_2}$ \; iff \; $(w_1,w_2)\in K$, $\displaystyle
\|w_i\|= \rho_i$ for some $i\in \{1,2\}$ and $\displaystyle
\|w_j\| \le \rho_j$ for $j\neq i$.
\end{enumerate}
\end{lem}

We utilize the following results from~\cite{jw-gi-jlms} regarding order preserving matrices:
\begin{defn}
A $\;2 \times 2$ matrix $\mathcal{Q}$ is said to be order
preserving (or positive) if $p_{1}\geq p_{0}$, $q_{1}\geq q_{0}$
imply
\begin{equation*}
\mathcal{Q}
\begin{pmatrix}
  p_{1} \\
  q_{1}
\end{pmatrix}%
\geq \mathcal{Q}
\begin{pmatrix}
  p_{0} \\
  q_{0}
\end{pmatrix},
\end{equation*}
in the sense of components.
\end{defn}
\begin{lem}\label{lematrix2}
Let
\begin{equation*}
\mathcal{Q}=
\begin{pmatrix}
  a & -b \\
  -c & d
\end{pmatrix}
\end{equation*}
with $a,b,c,d\geq 0$ and $\det \mathcal{Q}> 0$. Then
$\mathcal{Q}^{-1}$ is order preserving.
\end{lem}
\begin{rem} \label{rem1}
It is a consequence of Lemma \ref{lematrix2} that if
 \begin{equation*}
\mathcal{N}=
\begin{pmatrix}
  1-a & -b\\
  -c & 1-d
\end{pmatrix},
\end{equation*}
 satisfies the hypotheses of Lemma \ref{lematrix2},
$p \geq 0,\, q \geq 0$ and $\mu>1$ then
$$\mathcal{N}_\mu^{-1}\begin{pmatrix}
  p \\
  q
\end{pmatrix}\leq \mathcal{N}^{-1}\begin{pmatrix}
  p \\
  q
\end{pmatrix},$$
where
 \begin{equation*}
\mathcal{N}_\mu=
\begin{pmatrix}
  \mu-a & -b\\
  -c & \mu-d
\end{pmatrix}.
\end{equation*}
\end{rem}

With these tools we are able to prove a result concerning the set $K_{\rho_1,\rho_2}$.
\begin{lem}\label{sysind1}
Assume that
\begin{enumerate}
\item[$(\mathrm{I}_{\rho_1,\rho_2 }^{1})$]  there exist $\rho_1,\rho_2 >0$,  linear functionals $\alpha^{\rho_1,\rho_2}_{ij1}[\cdot]:\tilde{K_{1}}\rightarrow [0,+\infty)$, $\alpha^{\rho_1,\rho_2}_{ij2}[\cdot]:
  \tilde{K_{2}}\rightarrow [0,+\infty)$ given by
$$
\alpha^{\rho_1,\rho_2}_{ijl}[u]=\int_0^1 u(t)\,dA_{ijl}(t)
$$
 such that
\begin{itemize}
\item for $i,j,l=1,2,$ $dA_{ijl}$ is a positive Stieltjes measure,
\item for $i,j,l=1,2,$ $\alpha^{\rho_1,\rho_2}_{ijl}[\gamma_{ij}]<1$,
\item  for $i=1,2,$ $D_i:=(1-\alpha^{\rho_1,\rho_2}_{i1i}[\gamma_{i1}])(1-\alpha^{\rho_1,\rho_2}_{i2i}[\gamma_{i2}])
-\alpha^{\rho_1,\rho_2}_{i1i}[\gamma_{i2}]\alpha^{\rho_1,\rho_2}_{i2i}[\gamma_{i1}]>
0$,
\item for $i,j=1,2,$ $H_{ij}[u,v]\leq \alpha^{\rho_1,\rho_2}_{ij1}[u]+\alpha^{\rho_1,\rho_2}_{ij2}[v]$ for any $(u,v)\in \partial
K_{\rho_1,\rho_2}\,$,
\item for $i,j,l=1,2,$ with $l\not=i,$ the following inequality holds:
\begin{multline}\label{ind1s}
 f_i^{\rho_1,\rho_2} \Bigl( \Bigl(\| \gamma _{i1}\| \theta_{i1}+
\| \gamma _{i2}\| \theta_{i3} \Bigr)\int_0^1 \mathcal{K}_{i1i}(s)g_i(s)\,ds\\
+ \Bigl(\| \gamma _{i1}\|\theta_{i2}+
\| \gamma _{i2}\|\theta_{i4} \Bigr)\int_0^1 \mathcal{K}_{i2i}(s)g_i(s)\,ds +\dfrac{1}{m_i}\Bigr)\\
+\| \gamma _{i1}\|(\theta_{i1} Q_i+\theta_{i2} S_i)+ \| \gamma
_{i2}\|(\theta_{i3} Q_i+\theta_{i4}
S_i)+\frac{\rho_{l}}{\rho_i}\sum_{j=1,2}
||\gamma_{ij}||\alpha^{\rho_1,\rho_2}_{ijl}[1]  <1,
\end{multline}{}
where
\begin{align*}
f_{i}^{\rho_1,\rho_2}:=&\mathrm{ess }\sup
\Bigl\{\frac{f_{i}(t,u,v)}{\rho_i }:\;(t,u,v)\in
\lbrack 0,1]\times [ -\rho_1,\rho_1 ]\times [ -\rho_2,\rho_2 ]\Bigr\},\\
\frac{1}{m_{i}}:=&\sup_{t\in \lbrack 0,1]}\int_{0}^{1}|k_{i}(t,s)|g_{i}(s)\,ds,\,\,\,\mathcal{K}_{iji}(s):=\int_0^1 k_i(t,s) \,dA_{iji}(t),\\
Q_i:=&\frac{\rho_{l}}{\rho_i}\sum_{j=1,2}\alpha^{\rho_1,\rho_2}_{i1i}\left[\gamma_{ij}\right]\alpha^{\rho_1,\rho_2}_{ijl}
[1],\,\,
  S_i:=  \frac{\rho_{l}}{\rho_i}\sum_{j=1,2}\alpha^{\rho_1,\rho_2}_{i2i}\left[\gamma_{ij}\right]\alpha^{\rho_1,\rho_2}_{ijl}
[1],\\
\theta_{i1}:=&\frac{1-\alpha^{\rho_1,\rho_2}_{i2i}[\gamma_{i2}]}{D_i},\,\,
  \theta_{i2}:=\frac{ \alpha^{\rho_1,\rho_2}_{i1i}[\gamma_{i2}]}{D_i} ,\,\, \theta_{i3}:=\frac{\alpha^{\rho_1,\rho_2}_{i2i}[\gamma_{i1}]}{D_i},\,\,
 \theta_{i4}:=\frac{1-\alpha^{\rho_1,\rho_2}_{i1i}[\gamma_{i1}]}{D_i}.
\end{align*}
\end{itemize}
\end{enumerate}
Then  $i_{K}(T,K_{\rho_1,\rho_2})$ is equal to 1.
\end{lem}

\begin{proof}
We show that $\mu (u,v)\neq T(u,v)$ for every $(u,v)\in \partial K_{\rho_1,\rho_2 }$
and for every $\mu \geq 1$.
In fact, if this does not happen, there exist $\mu \geq 1$ and $(u,v)\in
\partial K_{\rho_1,\rho_2 }$ such that $\mu (u,v)=T(u,v)$. Assume, without loss of
generality, that $\| u\| =\rho_1 $ and $\| v\| \leq \rho_2 $. Then we have, for $t\in[0,1]$,
\begin{equation} \label{dis2}
\mu u(t)=\sum_{j=1,2}\gamma_{1j}(t)H_{1j} [u,v]+F_{1}(u,v)(t).
\end{equation}

Applying $\alpha^{\rho_1,\rho_2} _{111}$ and $\alpha^{\rho_1,\rho_2}_{121}$ to both sides of \eqref{dis2} gives
\begin{align*}
\mu\alpha^{\rho_1,\rho_2}_{111}[u]
 &\leq
\sum_{j=1,2}\alpha^{\rho_1,\rho_2}_{111}\left[\gamma_{1j}\right]\left(\alpha^{\rho_1,\rho_2}_{1j1}[u]+\alpha^{\rho_1,\rho_2}_{1j2}[v]\right)+\alpha^{\rho_1,\rho_2}_{111}
\left[F_{1}(u,v)\right] ;\\
\mu\alpha^{\rho_1,\rho_2}_{121}[u]& \leq
\sum_{j=1,2}\alpha^{\rho_1,\rho_2}_{121}\left[\gamma_{1j}\right]\left(\alpha^{\rho_1,\rho_2}_{1j1}[u]+\alpha^{\rho_1,\rho_2}_{1j2}[v]\right)+\alpha^{\rho_1,\rho_2}_{121}\left[F_{1}(u,v)\right] .
\end{align*}
Since $v(t)\leq \rho_2$ for all $t\in \lbrack 0,1]$, we obtain
\begin{align*}
\mu\alpha^{\rho_1,\rho_2}_{111}[u]&\leq
\sum_{j=1,2}\alpha^{\rho_1,\rho_2}_{111}\left[\gamma_{1j}\right]\alpha^{\rho_1,\rho_2}_{1j1}
[u]+\rho_2
\sum_{j=1,2}\alpha^{\rho_1,\rho_2}_{111}\left[\gamma_{1j}\right
]\alpha^{\rho_1,\rho_2}_{1j2}
[1]+\alpha^{\rho_1,\rho_2}_{111}[F_{1}(u,v)]; \\
\mu\alpha^{\rho_1,\rho_2}_{121}[u]&\leq \sum_{j=1,2}\alpha^{\rho_1,\rho_2}_{121}\left[\gamma_{1j}\right]\alpha^{\rho_1,\rho_2}_{1j1}
[u]+\rho_2
\sum_{j=1,2}\alpha^{\rho_1,\rho_2}_{121}\left[\gamma_{1j}\right]
\alpha^{\rho_1,\rho_2}_{1j2}
[1]+\alpha^{\rho_1,\rho_2}_{121}[F_{1}(u,v)].
\end{align*}
Thus we have
\begin{multline}\label{idx1in}
\begin{pmatrix}
\mu- \alpha^{\rho_1,\rho_2}_{111}\left[\gamma_{11}\right] & -\alpha^{\rho_1,\rho_2}_{111}\left[\gamma_{12}\right]\\
-\alpha^{\rho_1,\rho_2}_{121}\left[\gamma_{11}\right] & \mu- \alpha^{\rho_1,\rho_2}_{121}\left[\gamma_{12}\right]
\end{pmatrix}
\begin{pmatrix}
\alpha^{\rho_1,\rho_2}_{111}[u]\\
\alpha^{\rho_1,\rho_2}_{121}[u]
\end{pmatrix}\\
\leq
\begin{pmatrix}
\rho_1
\displaystyle\sum_{j=1,2}\frac{\rho_2}{\rho_1}\alpha^{\rho_1,\rho_2}_{111}\left[\gamma_{1j}\right]\alpha^{\rho_1,\rho_2}_{1j2}
[1]+\alpha^{\rho_1,\rho_2}_{111}[F_{1}(u,v)]\\
\rho_1
\displaystyle\sum_{j=1,2}\frac{\rho_2}{\rho_1}\alpha^{\rho_1,\rho_2}_{121}\left[\gamma_{1j}\right]\alpha^{\rho_1,\rho_2}_{1j2}
[1]+\alpha^{\rho_1,\rho_2}_{121}[F_{1}(u,v)]
\end{pmatrix}.
\end{multline}
The matrix
$$
\mathcal{M}_{\mu}=
\begin{pmatrix}
\mu- \alpha^{\rho_1,\rho_2}_{111}\left[\gamma_{11}\right] & -\alpha^{\rho_1,\rho_2}_{111}\left[\gamma_{12}\right]\\
-\alpha^{\rho_1,\rho_2}_{121}\left[\gamma_{11}\right] & \mu- \alpha^{\rho_1,\rho_2}_{121}\left[\gamma_{12}\right]
\end{pmatrix}
$$
satisfies the hypotheses of Lemma \ref{lematrix2}, thus $(\mathcal{M}_{\mu})^{-1}$ is order preserving.
If we apply
$(\mathcal{M}_{\mu})^{-1}$ to both sides of the inequality \eqref{idx1in}  we
obtain
\begin{multline*}
\begin{pmatrix}
\alpha^{\rho_1,\rho_2}_{111}[u]\\
\alpha^{\rho_1,\rho_2}_{121}[u]
\end{pmatrix}
\leq
\frac{1}{\det(\mathcal{M}_{\mu})}
\begin{pmatrix}
\mu- \alpha^{\rho_1,\rho_2}_{121}\left[\gamma_{12}\right] & \alpha^{\rho_1,\rho_2}_{111}\left[\gamma_{12}\right]\\
\alpha^{\rho_1,\rho_2}_{121}\left[\gamma_{11}\right] & \mu- \alpha^{\rho_1,\rho_2}_{111}\left[\gamma_{11}\right]
\end{pmatrix}
\\
\times
\begin{pmatrix}
\rho_1
\displaystyle\sum_{j=1,2}\frac{\rho_2}{\rho_1}\alpha^{\rho_1,\rho_2}_{111}\left[\gamma_{1j}\right]\alpha^{\rho_1,\rho_2}_{1j2}
[1]+\alpha^{\rho_1,\rho_2}_{111}[F_{1}(u,v)]\\
\rho_1
\displaystyle\sum_{j=1,2}\frac{\rho_2}{\rho_1}\alpha^{\rho_1,\rho_2}_{121}\left[\gamma_{1j}\right]\alpha^{\rho_1,\rho_2}_{1j2}
[1]+\alpha^{\rho_1,\rho_2}_{121}[F_{1}(u,v)]
\end{pmatrix},
\end{multline*}
and by Remark \ref{rem1}, we have
\begin{multline*}
\begin{pmatrix}
\alpha^{\rho_1,\rho_2}_{111}[u]\\
\alpha^{\rho_1,\rho_2}_{121}[u]
\end{pmatrix}
\leq
\frac{1}{D_1}
\begin{pmatrix}
1- \alpha^{\rho_1,\rho_2}_{121}\left[\gamma_{12}\right] & \alpha^{\rho_1,\rho_2}_{111}\left[\gamma_{12}\right]\\
\alpha^{\rho_1,\rho_2}_{121}\left[\gamma_{11}\right] & 1- \alpha^{\rho_1,\rho_2}_{111}\left[\gamma_{11}\right]
\end{pmatrix}
\\
\times
\begin{pmatrix}
\rho_1
\displaystyle\sum_{j=1,2}\frac{\rho_2}{\rho_1}\alpha^{\rho_1,\rho_2}_{111}\left[\gamma_{1j}\right]\alpha^{\rho_1,\rho_2}_{1j2}
[1]+\alpha^{\rho_1,\rho_2}_{111}[F_{1}(u,v)]\\
\rho_1
\displaystyle\sum_{j=1,2}\frac{\rho_2}{\rho_1}\alpha^{\rho_1,\rho_2}_{121}\left[\gamma_{1j}\right]\alpha^{\rho_1,\rho_2}_{1j2}
[1]+\alpha^{\rho_1,\rho_2}_{121}[F_{1}(u,v)]
\end{pmatrix},
\end{multline*}
that is
$$
\begin{pmatrix}
\alpha^{\rho_1,\rho_2}_{111}[u]\\
\alpha^{\rho_1,\rho_2}_{121}[u]
\end{pmatrix}
\leq
\begin{pmatrix}
\theta_{11} & \theta_{12}\\
\theta_{13} & \theta_{14}
\end{pmatrix}
\begin{pmatrix}
\rho_1 Q_1+\alpha^{\rho_1,\rho_2}_{111}[F_1(u,v)]\\
\rho_1 S_1+\alpha^{\rho_1,\rho_2}_{121}[F_1(u,v)]
\end{pmatrix}.
$$
Thus we have
\begin{equation}\label{disb}
\begin{pmatrix}
\alpha^{\rho_1,\rho_2}_{111}[u]\\
\alpha^{\rho_1,\rho_2}_{121}[u]
\end{pmatrix}
\leq
\begin{pmatrix}
\rho_1(\theta_{11} Q_1+\theta_{12} S_1)+\theta_{11}\alpha^{\rho_1,\rho_2}_{111}[F_1(u,v)]+\theta_{12}\alpha^{\rho_1,\rho_2}_{121}[F_1(u,v)]\\
\rho_1(\theta_{13} Q_1+\theta_{14} S_1)+\theta_{13}\alpha^{\rho_1,\rho_2}_{111}[F_1(u,v)]+\theta_{14}\alpha^{\rho_1,\rho_2}_{121}[F_1(u,v)]
\end{pmatrix}.
\end{equation}
Since we have
\begin{gather}
\begin{aligned}\label{dis2a}
\mu |u(t)|\leq  \sum_{j=1,2}|\gamma_{1j}(t)|\alpha^{\rho_1,\rho_2}_{1j1}[u] &+\rho_1\sum_{j=1,2} \frac{\rho_2}{\rho_1}|\gamma_{1j}(t)|\alpha^{\rho_1,\rho_2}_{1j2}[1]\\
&+\int_{0}^{1}|k_{1}(t,s)|g_{1}(s)f(s,u(s),v(s))\,ds,
\end{aligned}
\end{gather}
substituting \eqref{disb} into \eqref{dis2a} gives
\begin{align*}
\mu |u(t)| \leq &\rho_1 \Bigl(|\gamma_{11}(t)|(\theta_{11} Q_1+\theta_{12} S_1)+
|\gamma_{12}(t)|(\theta_{13} Q_1+\theta_{14} S_1)+\sum_{j=1,2} \frac{\rho_2}{\rho_1}|\gamma_{1j}(t)|\alpha^{\rho_1,\rho_2}_{1j2}[1] \Bigr) \\
&+\Bigl(|\gamma_{11}(t)|\theta_{11}+
|\gamma_{12}(t)|\theta_{13} \Bigr)\int_0^1 \mathcal{K}_{111}(s)g_1(s)f_1(s,u(s),v(s))\,ds \\
&+\Bigl(|\gamma_{11}(t)|\theta_{12}+
|\gamma_{12}(t)|\theta_{14} \Bigr)\int_0^1 \mathcal{K}_{121}(s)g_1(s)f_1(s,u(s),v(s))\,ds\\
&+\int_{0}^{1}|k_{1}(t,s)|g_{1}(s)f_1(s,u(s),v(s))\,ds.
\end{align*}

Taking the supremum over $[0,1]$ gives
\begin{align*}
\mu {\rho_1} \leq &\rho_1 \Bigl(\| \gamma _{11}\| (\theta_{11} Q_1+\theta_{12} S_1)+
\| \gamma _{12}\|(\theta_{13} Q_1+\theta_{14} S_1)+\sum_{j=1,2} \frac{\rho_2}{\rho_1}||\gamma_{1j}||\alpha^{\rho_1,\rho_2}_{1j2}[1]   \Bigr) \\
&+{\rho_1} f_{1}^{\rho_1,\rho_2}\Bigl(\| \gamma _{11}\| \theta_{11}+
\| \gamma _{12}\| \theta_{13} \Bigr)\int_0^1 \mathcal{K}_{111}(s)g_1(s)\,ds \\
&+{\rho_1} f_{1}^{\rho_1,\rho_2}\Bigl(\| \gamma _{11}\|\theta_{12}+
\| \gamma _{12}\|\theta_{14} \Bigr)\int_0^1 \mathcal{K}_{121}(s)g_1(s)\,ds+{\rho_1} f_{1}^{\rho_1,\rho_2} \dfrac{1}{m_1}.
\end{align*}

Using the hypothesis \eqref{ind1s} we obtain $\mu \rho_1 <\rho_1 .$ This
contradicts the fact that $\mu \geq 1$ and proves the result.
\end{proof}

We give a first Lemma which shows that the index is 0 on a set
$V_{\rho_1,\rho_2}$.

\begin{lem}\label{idx0n}
Assume that
\begin{enumerate}
\item[$(\mathrm{I}_{\rho_1,\rho_2 }^{0})$]  there exist $\rho_1,\rho_2 >0$,  linear functionals $\alpha^{\rho_1,\rho_2}_{ij1}[\cdot]:\tilde{K_{1}}\rightarrow [0,+\infty)$, $\alpha^{\rho_1,\rho_2}_{ij2}[\cdot]:
  \tilde{K_{2}}\rightarrow [0,+\infty)$ given by
$$
\alpha^{\rho_1,\rho_2}_{ijl}[u]=\int_0^1 u(t)\,dA_{ijl}(t)
$$
 such that  for every $i,j,l=1,2$
\begin{itemize}
\item $dA_{ijl}$ is a positive Stieltjes measure,
\item $\alpha^{\rho_1,\rho_2}_{ijl}[\gamma_{ij}]<1$,
\item  $D_i:=(1-\alpha^{\rho_1,\rho_2}_{i1i}[\gamma_{i1}])(1-\alpha^{\rho_1,\rho_2}_{i2i}[\gamma_{i2}])
-\alpha^{\rho_1,\rho_2}_{i1i}[\gamma_{i2}]\alpha^{\rho_1,\rho_2}_{i2i}[\gamma_{i1}]>
0$,
\item $H_{ij}[u,v]\geq \alpha^{\rho_1,\rho_2}_{ij1}[u]+\alpha^{\rho_1,\rho_2}_{ij2}[v]$ for any $(u,v)\in \partial
V_{\rho_1,\rho_2}\,$,
\item the following inequality holds:
\begin{multline}\label{eq0s}
f_{i,(\rho_1
,{\rho_2})}\Biggl(\Bigl(\dfrac{c_{i1}\|\gamma_{i1}\|}{{D}_i}(1-\alpha^{\rho_1,\rho_2}_{i2i}[\gamma_{i2}])+
\dfrac{c_{i2}\|\gamma_{i2}\|}{{D}_i}\alpha^{\rho_1,\rho_2}_{i2i}[\gamma_{i1}]\Bigr) \int_{a_{i}}^{b_{i}}\mathcal{K}_{i1i}(s)g_{i}(s)\,ds\\
+\Bigl(\dfrac{c_{i1}\|\gamma_{i1}\|}{{D}_i}\alpha^{\rho_1,\rho_2}_{i1i}[\gamma_{i2}]+
\dfrac{c_{i2}\|\gamma_{i2}\|}{{D}_i}(1-\alpha^{\rho_1,\rho_2}_{i1i}[\gamma_{i1}])\Bigr)\int_{a_i}^{b_i}
 \mathcal{K}_{i2i}(s)g_i(s)\,ds +\frac{1}{M_{i}}\Biggr)>1,
\end{multline}{}
where
\begin{align*}
f_{1,(\rho_1,\rho_2)}:=&\mathrm{ess }\inf \Bigl\{ \frac{f_1(t,u,v)}{ \rho_1}:\; (t,u,v)\in [a_1,b_1]\times[\rho_1,\rho_1/c_1]\times[-\rho_2/c_2, \rho_2/c_2]\Bigr\},\\
f_{2,(\rho_1\rho_2)}:=&\mathrm{ess }\inf \Bigl\{ \frac{f_2(t,u,v)}{ \rho_2}:\; (t,u,v)\in [a_2,b_2]\times[-\rho_1/c_1,\rho_1/c_1]\times[\rho_2, \rho_2/c_2]\Bigr\},\\
 \frac{1}{M_i}:=&\inf_{t\in
[a_i,b_i]}\int_{a_i}^{b_i} k_i(t,s) g_i(s)\,ds.
\end{align*}
\end{itemize}
\end{enumerate}
Then $i_{K}(T,V_{\rho_1,\rho_2})=0$.
\end{lem}

\begin{proof}
 Note that the constant function $(1,1)$ belongs to $K$. We prove that $(u,v)\ne
T(u,v)+\lambda (1,1)$ for  $(u,v)\in \partial
V_{\rho_1,\rho_2}$ and  $\lambda \geq 0$.\\
In fact, if this does not happen, there exist $(u,v)\in \partial V_{\rho_1,\rho_2}$ and
$\lambda\geq 0$ such that $(u,v)=T(u,v)+\lambda(1,1)$.
Without loss of generality, we can assume that for $t\in [a_1,b_1]$ we have
$\rho_1\leq u(t)\leq {\rho_1/c_1}$,  $\min u(t)=\rho_1$ and  $-\rho_2/c_2\leq v(t)\leq {\rho_2/c_2}$.

Then, for $t\in [a_1,b_1]$, we obtain
\begin{equation}\label{diss1n}
u(t)= \sum_{j=1,2}\gamma_{1j}(t)H_{1j}[u,v]+F_{1}(u,v)(t)+
\lambda\,1\,.
\end{equation}
Applying $\alpha^{\rho_1,\rho_2}_{1l1}$, $l=1,2$, to both sides of \eqref{diss1n} gives
\begin{align*}
\alpha^{\rho_1,\rho_2}_{1l1}[u]&\geq
\sum_{j=1,2}\alpha^{\rho_1,\rho_2}_{1l1}\left[\gamma_{1j}\right]\alpha^{\rho_1,\rho_2}_{1j1}[u]+\alpha^{\rho_1,\rho_2}_{1l1}\left[F_{1}(u,v)\right]+
\lambda \alpha^{\rho_1,\rho_2}_{1l1}[1 ].
\end{align*}
Thus we have
\begin{multline*}
\begin{pmatrix}
1-\alpha^{\rho_1,\rho_2}_{111}[\gamma_{11}] & -\alpha^{\rho_1,\rho_2}_{111}[\gamma_{12}]\\
-\alpha^{\rho_1,\rho_2}_{121}[\gamma_{11}] & 1-\alpha^{\rho_1,\rho_2}_{121}[\gamma_{12}]
\end{pmatrix}
\begin{pmatrix}
\alpha^{\rho_1,\rho_2}_{111}[u]\\
\alpha^{\rho_1,\rho_2}_{121}[u]
\end{pmatrix}
\\
\geq
\begin{pmatrix}
\alpha^{\rho_1,\rho_2}_{111}[F_1(u,v)]+\lambda \alpha^{\rho_1,\rho_2}_{111} [1]\\
\alpha^{\rho_1,\rho_2}_{121}[F_1(u,v)]+ \lambda\alpha_{121}[1]
\end{pmatrix}
\geq
\begin{pmatrix}
\alpha^{\rho_1,\rho_2}_{111}[F_1(u,v)]\\
\alpha^{\rho_1,\rho_2}_{121}[F_1(u,v)]
\end{pmatrix}.
\end{multline*}
In a similar way as in the proof of Lemma \ref{sysind1}, via order preserving matrices, we
obtain
$$
\begin{pmatrix}
\alpha^{\rho_1,\rho_2}_{111}[u]\\
\alpha^{\rho_1,\rho_2}_{121}[u]
\end{pmatrix}
\geq
\frac{1}{{D}_1}
\begin{pmatrix}
1-\alpha^{\rho_1,\rho_2}_{121}[\gamma_{12}]&\alpha^{\rho_1,\rho_2}_{111}[\gamma_{12}]\\
\alpha^{\rho_1,\rho_2}_{121}[\gamma_{11}] & 1-\alpha^{\rho_1,\rho_2}_{111}[\gamma_{11}]
\end{pmatrix}
\begin{pmatrix}
\alpha^{\rho_1,\rho_2}_{111}[F_1(u,v)]\\
\alpha^{\rho_1,\rho_2}_{121}[F_1(u,v)]
\end{pmatrix}\,.
$$
We have, for $t\in [a_1,b_1]$,
\begin{align*} u(t)
\geq&\Bigl(\dfrac{c_{11}\|\gamma_{11}\|}{{D}_1}(1-\alpha^{\rho_1,\rho_2}_{121}[\gamma_{12}])+
\dfrac{c_{12}\|\gamma_{12}\|}{{D}_1}\alpha^{\rho_1,\rho_2}_{121}[\gamma_{11}]\Bigr)\int_{a_1}^{b_1} \mathcal{K}_{111}(s)g_1(s)f_1(s,u(s),v(s))\,ds\\
&+\Bigl(\dfrac{c_{11}\|\gamma_{11}\|}{{D}_1}\alpha^{\rho_1,\rho_2}_{111}[\gamma_{12}]+
\dfrac{c_{12}\|\gamma_{12}\|}{{D}_1}(1-\alpha^{\rho_1,\rho_2}_{111}[\gamma_{11}])\Bigr) \int_{a_1}^{b_1} \mathcal{K}_{121}(s)g_1(s)f_1(s,u(s),v(s))\,ds\\
&+  \int_{a_1}^{b_1} k_{1}(t,s)g_1(s)f_1(s,u(s),v(s))\,ds + \lambda.
\end{align*}
Taking the minimum over $[a_{1},b_{1}]$ and
using the hypothesis \eqref{eq0s} we obtain $\rho_1>\rho_1 +\lambda $, a contradiction.
\end{proof}

The following Lemma provides a result of index 0 on
$V_{\rho_1,\rho_2}$ of a different flavour; here we control the growth
of just one nonlinearity $f_i$, at the cost of having to deal with
a larger domain. We mention that nonlinearities with different
growth were studied also in ~\cite{gifmpp-cnsns, gipp-nonlin,
precup1,precup2,ya1}.
\begin{lem}
Assume that
\begin{enumerate}
\item[$(\mathrm{I}_{\rho_1,\rho_2 }^{0})^{\diamond}$]  there exist $\rho_1,\rho_2 >0$,  linear functionals $\alpha^{\rho_1,\rho_2}_{ij1}[\cdot]:\tilde{K_{1}}\rightarrow [0,+\infty)$, $\alpha^{\rho_1,\rho_2}_{ij2}[\cdot]:
  \tilde{K_{2}}\rightarrow [0,+\infty)$ given by
$$
\alpha^{\rho_1,\rho_2}_{ijl}[u]=\int_0^1 u(t)\,dA_{ijl}(t)
$$
 such that for almost one $i=1,2$ and for every $j,l=1,2$
\begin{itemize}
\item $dA_{ijl}$ is a positive Stieltjes measure,
\item $\alpha^{\rho_1,\rho_2}_{ijl}[\gamma_{ij}]<1$,
\item  $D_i:=(1-\alpha^{\rho_1,\rho_2}_{i1i}[\gamma_{i1}])(1-\alpha^{\rho_1,\rho_2}_{i2i}[\gamma_{i2}])
-\alpha^{\rho_1,\rho_2}_{i1i}[\gamma_{i2}]\alpha^{\rho_1,\rho_2}_{i2i}[\gamma_{i1}]>
0$,
\item $H_{ij}[u,v]\geq \alpha^{\rho_1,\rho_2}_{ij1}[u]+\alpha^{\rho_1,\rho_2}_{ij2}[v]$ for any $(u,v)\in \partial
V_{\rho_1,\rho_2}$,
\item the following inequality holds:
\begin{multline}\label{diamante}
f^{\diamond}_{i,(\rho_1
,{\rho_2})}\Biggl(\Bigl(\dfrac{c_{i1}\|\gamma_{i1}\|}{{D}_i}(1-\alpha^{\rho_1,\rho_2}_{12i}[\gamma_{i2}])+
\dfrac{c_{i2}\|\gamma_{i2}\|}{{D}_i}\alpha^{\rho_1,\rho_2}_{12i}[\gamma_{i1}]\Bigr) \int_{a_{i}}^{b_{i}}\mathcal{K}_{i1i}(s)g_{i}(s)\,ds\\
+\Bigl(\dfrac{c_{i1}\|\gamma_{i1}\|}{{D}_i}\alpha^{\rho_1,\rho_2}_{i1i}[\gamma_{i2}]+
\dfrac{c_{i2}\|\gamma_{i2}\|}{{D}_i}(1-\alpha^{\rho_1,\rho_2}_{i1i}[\gamma_{i1}])\Bigr)\int_{a_i}^{b_i}
 \mathcal{K}_{i2i}(s)g_i(s)\,ds +\frac{1}{M_{i}}\Biggr)>1,
\end{multline}{}
\end{itemize}
\end{enumerate}
where
\begin{equation*}
f^{\diamond}_{1,(\rho_1 ,{\rho_2})}:=\mathrm{ess }\inf \Bigl\{
\frac{f_1(t,u,v)}{ \rho_1}:\; (t,u,v)\in
[a_1,b_1]\times[0,\rho_1/c]\times[-\rho_2/c_2, \rho_2/c_2]\Bigr\},
\end{equation*}
\begin{equation*}
f^{\diamond}_{2,(\rho_1 ,{\rho_2})}:=\mathrm{ess }\inf \Bigl\{
\frac{f_2(t,u,v)}{ \rho_2}:\; (t,u,v)\in
[a_2,b_2]\times[-\rho_1/c_1,\rho_1/c_1]\times[0,
\rho_2/c_2]\Bigr\}.
\end{equation*}
Then $i_{K}(T,V_{\rho_1,\rho_2})=0$.
\end{lem}
\begin{proof}
Suppose that the condition \eqref{diamante} holds for $i=1$. Let
$(u,v)\in \partial V_{\rho_1,\rho_2 }$ and $\lambda \geq 0$ such
that $(u,v)=T(u,v)+\lambda (1,1)$. Thus for $t\in [a_1,b_1]$ we
have $\min u(t)\leq \rho_1$, $0 \leq u(t)\leq \rho_1/c_1$ and
$-\rho_2/c_2\leq v(t)\leq \rho_2/c_2$. Proceeding in a similar proof of Lemma
\ref{idx0n}, we obtain a contradiction.
\end{proof}
\begin{rem}
In the case of $[a_1,b_1]=[a_2,b_2]$ we can relax the assumptions on the nonlinearities $f_i$,   providing a modification of the conditions   $(\mathrm{I}_{\rho
_{1},\rho_2}^{0})$ and $(\mathrm{I}_{\rho _{1},\rho_2}^{0})^{\diamond}$, similar to the one in \cite{gipp-nodea}. We omit the statement of these results.
\end{rem}
We now state  results regarding the existence of at least one, two
or three \emph{nontrivial} solutions of the system \eqref{syst}.

\begin{thm}\label{mult-sys}
The system \eqref{syst} has at least one nontrivial solution
in $K$ if one of the following conditions holds.

\begin{enumerate}

\item[$(S_{1})$]  For $i=1,2$ there exist $\rho _{i},r _{i}\in (0,+\infty )$ with $\rho
_{i}/c_i<r _{i}$ such that $(\mathrm{I}_{\rho _{1},\rho_2}^{0})\;\;[\text{or}\;(%
\mathrm{I}_{\rho _{1},\rho_2}^{0})^{\diamond}]$, $(\mathrm{I}_{r _{1},r_2}^{1})$ hold.

\item[$(S_{2})$] For $i=1,2$ there exist $\rho _{i},r _{i}\in (0,+\infty )$ with $\rho
_{i}<r _{i}$ such that $(\mathrm{I}_{\rho _{1},\rho_2}^{1}),\;\;(\mathrm{I}%
_{r _{1},r_2}^{0})$ hold.
\end{enumerate}

The system \eqref{syst} has at least two nontrivial solutions in $K$ if one of the following conditions holds.

\begin{enumerate}

\item[$(S_{3})$] For $i=1,2$ there exist $\rho _{i},r _{i},s_i\in (0,+\infty )$
with $\rho _{i}/c_i<r_i <s _{i}$ such that $(\mathrm{I}_{\rho
_{1},\rho_2}^{0})$, $[\text{or}\;(\mathrm{I}_{\rho _{1},\rho_2}^{0})^{\diamond}],\;\;(%
\mathrm{I}_{r _{1},r_2}^{1})$ $\text{and}\;\;(\mathrm{I}_{s _{1},s_2}^{0})$
hold.

\item[$(S_{4})$] For $i=1,2$ there exist $\rho _{i},r _{i},s_i\in (0,+\infty )$
with $\rho _{i}<r _{i}$ and $r _{i}/c_i<s _{i}$ such that $(\mathrm{I}%
_{\rho _{1},\rho_2}^{1}),\;\;(\mathrm{I}_{r _{1},r_2}^{0})$ $\text{and}\;\;(\mathrm{I%
}_{s _{1},s_2}^{1})$ hold.
\end{enumerate}

The system \eqref{syst} has at least three nontrivial solutions in $K$ if one
of the following conditions holds.

\begin{enumerate}
\item[$(S_{5})$] For $i=1,2$ there exist $\rho _{i},r _{i},s_i,\sigma_i\in
(0,+\infty )$ with $\rho _{i}/c_i<r _{i}<s _{i}$ and $s _{i}/c_i<\sigma
_{i}$ such that $(\mathrm{I}_{\rho _{1},\rho_2}^{0})\;\;[\text{or}\;(\mathrm{I}%
_{\rho _{1},\rho_2}^{0})^{\diamond}],$ $(\mathrm{I}_{r _{1},r_2}^{1}),\;\;(\mathrm{I}%
_{s_1,s_2}^{0})\;\;\text{and}\;\;(\mathrm{I}_{\sigma _{1},\sigma_2}^{1})$ hold.

\item[$(S_{6})$] For $i=1,2$ there exist $\rho _{i},r _{i},s_i,\sigma_i\in
(0,+\infty )$ with $\rho _{i}<r _{i}$ and $r _{i}/c_i<s _{i}<\sigma _{i}$
such that $(\mathrm{I}_{\rho _{1},\rho_2}^{1}),\;\;(\mathrm{I}_{r
_{1},r_2}^{0}),\;\;(\mathrm{I}_{s _{1},s_2}^{1})$ $\text{and}\;\;(\mathrm{I}%
_{\sigma _{1},\sigma_2}^{0})$ hold.
\end{enumerate}
\end{thm}

\subsection{Non-existence results for system of perturbed  integral equations }
We now prove some non-existence results  for systems when both the functions $\gamma_{ij}$ and the kernels $k_i$ are allowed to change sign.

\begin{thm}\label{nonexi1s}
Assume that there exist linear functionals
$\alpha_{ij}[\cdot]:\tilde{K_{i}}\rightarrow [0,+\infty)$ given by
  $$
\alpha^{\rho_1,\rho_2}_{ij}[u]=\int_0^1 u(t)\,dA_{ij}(t)
$$ such that  for $i,j=1,2$
\begin{itemize}
\item $dA_{ij}$ is a positive Stieltjes measure,
\item $\alpha_{ij}[\gamma_{ij}]<1$,
\item  $D_i:=(1-\alpha_{i1}[\gamma_{i1}])(1-\alpha_{i2}[\gamma_{i2}])
-\alpha_{i1}[\gamma_{i2}]\alpha_{i2}[\gamma_{i1}]>
0$,
\item $H_{ij}[u_1,u_2]\leq \alpha_{ij}[u_i]$ for every $(u_1,u_2)\in K$,
\item $f_i(t,u_1,u_2)<N_{i}|u_i|$ for every $t\in [0,1]$ and $u_i\neq 0$,
where
\begin{align*}
\frac{1}{N_{i}}:=\sup_{t \in [0,1] }\Bigl\{&\int_{0}^{1}|k_{i}(t,s)|g_{i}(s)\,ds+\Bigl(|\gamma_{i1}(t)|\dfrac{1- \alpha_{i2}\left[\gamma_{i2}\right]}{D_i}+ |\gamma_{i2}(t)|\dfrac{\alpha_{i2}\left[\gamma_{i1}\right]}{D_i}\Bigr)\int_0^1 \mathcal{K}_{i1}(s)g_i(s)\,ds \\
&+\Bigl(|\gamma_{i1}(t)|\dfrac{\alpha_{i1}\left[\gamma_{i2}\right]}{D_i}+
|\gamma_{i2}(t)|\dfrac{1- \alpha_{i1}\left[\gamma_{i1}\right]}{D_i} \Bigr)\int_0^1 \mathcal{K}_{i2}(s)g_i(s)\,ds\Bigr\}
\end{align*}
and
\begin{equation*}
\mathcal{K}_{ij}(s):=\int_0^1 k_i(t,s) \,dA_{ij}.
\end{equation*}
\end{itemize}
Then there is no nontrivial solution of the system \eqref{syst} in $K$.
\end{thm}

\begin{proof}
Suppose that there exists  $(u,v)\in K$ such that $ (u,v)=T(u,v)$ and  assume, without loss of
generality, that $|| u|| =\nu>0 $.  Then we have, for $t\in[0,1]$,
\begin{equation} \label{disn}
 u(t)=\sum_{j=1,2}\gamma_{1j}(t)H_{1j} [u,v]+F_{1}(u,v)(t).
\end{equation}
Applying $\alpha _{1l}$, $l=1,2$, to both sides of \eqref{disn} gives
\begin{align*}
\alpha_{1l}[u]&=
\sum_{j=1,2}\alpha_{1l}\left[\gamma_{1j}\right]H_{1j}[u,v]+\alpha_{1l}\left[F_{1}(u,v)\right]
 \leq \sum_{j=1,2}\alpha_{1l}\left[\gamma_{1j}\right]\alpha_{1j}[u]+\alpha_{1l}\left[F_{1}(u,v)\right].
\end{align*}
Thus we have
\begin{gather}
\begin{aligned}\label{idx1inn}
\begin{pmatrix}
1- \alpha_{11}\left[\gamma_{11}\right] & -\alpha_{11}\left[\gamma_{12}\right]\\
-\alpha_{12}\left[\gamma_{11}\right] & 1- \alpha_{12}\left[\gamma_{12}\right]
\end{pmatrix}
\begin{pmatrix}
\alpha_{11}[u]\\
\alpha_{12}[u]
\end{pmatrix}
\leq
\begin{pmatrix}
\alpha_{11}[F_{1}(u,v)]\\
\alpha_{12}[F_{1}(u,v)]
\end{pmatrix}.
\end{aligned}
\end{gather}
In a similar way as in the proof of Lemma \ref{sysind1}  we
obtain, via order preserving matrices,

\begin{equation}\label{disbn}
\begin{pmatrix}
\alpha_{11}[u]\\
\alpha_{12}[u]
\end{pmatrix}
\leq
\begin{pmatrix}
\dfrac{1- \alpha_{12}\left[\gamma_{12}\right]}{D_1}\alpha_{11}[F_1(u,v)]+\dfrac{\alpha_{11}\left[\gamma_{12}\right]}{D_1}\alpha_{12}[F_1(u,v)]\\
\dfrac{\alpha_{12}\left[\gamma_{11}\right]}{D_1}\alpha_{11}[F_1(u,v)]+\dfrac{1- \alpha_{11}\left[\gamma_{11}\right]}{D_1}\alpha_{12}[F_1(u,v)]
\end{pmatrix}.
\end{equation}
Since we have
\begin{gather}
\begin{aligned}\label{dis2an}
|u(t)|\leq  \sum_{j=1,2}|\gamma_{1j}(t)|\alpha_{1j}[u] +\int_{0}^{1}|k_{1}(t,s)|g_{1}(s)f(s,u(s),v(s))\,ds,
\end{aligned}
\end{gather}
substituting \eqref{disbn} into \eqref{dis2an}, we obtain, for $t\in [0,1]$,
\begin{align*}
 |u(t)| \leq &\Bigl(|\gamma_{11}(t)|\dfrac{1- \alpha_{12}\left[\gamma_{12}\right]}{D_1}+ |\gamma_{12}(t)|\dfrac{\alpha_{12}\left[\gamma_{11}\right]}{D_1}\Bigr)\int_0^1 \mathcal{K}_{11}(s)g_1(s)f_1(s,u(s),v(s))\,ds \\
&+\Bigl(|\gamma_{11}(t)|\dfrac{\alpha_{11}\left[\gamma_{12}\right]}{D_1}+
|\gamma_{12}(t)|\dfrac{1- \alpha_{11}\left[\gamma_{11}\right]}{D_1} \Bigr)\int_0^1 \mathcal{K}_{12}(s)g_1(s)f_1(s,u(s),v(s))\,ds\\
&+\int_{0}^{1}|k_{1}(t,s)|g_{1}(s)f_1(s,u(s),v(s))\,ds\\
< &\Bigl(|\gamma_{11}(t)|\dfrac{1- \alpha_{12}\left[\gamma_{12}\right]}{D_1}+ |\gamma_{12}(t)|\dfrac{\alpha_{12}\left[\gamma_{11}\right]}{D_1}\Bigr)\int_0^1 \mathcal{K}_{11}(s)g_1(s)N_1|u(s)|\,ds \\
&+\Bigl(|\gamma_{11}(t)|\dfrac{\alpha_{11}\left[\gamma_{12}\right]}{D_1}+
|\gamma_{12}(t)|\dfrac{1- \alpha_{11}\left[\gamma_{11}\right]}{D_1} \Bigr)\int_0^1 \mathcal{K}_{12}(s)g_1(s)N_1|u(s)|\,ds\\
&+\int_{0}^{1}|k_{1}(t,s)|g_{1}(s)N_1|u(s)|\,ds.
\end{align*}
Taking the supremum over $[0,1]$ gives
 $\nu <\nu $, a contradiction that proves the result.
\end{proof}

\begin{thm}\label{noexi2 n}
Assume that
 there exist   linear functionals $\alpha_{ij}[\cdot]:\tilde{K_{i}}\rightarrow [0,+\infty)$ given by
 $$
\alpha^{\rho_1,\rho_2}_{ij}[u]=\int_0^1 u(t)\,dA_{ij}(t)
$$ such that  for every $i,j=1,2$
\begin{itemize}
\item $dA_{ij}$ is a positive Stieltjes measure,
\item $\alpha_{ij}[\gamma_{ij}]<1$,
\item  $D_i:=(1-\alpha_{i1}[\gamma_{i1}])(1-\alpha_{i2}[\gamma_{i2}])
-\alpha_{i1}[\gamma_{i2}]\alpha_{i2}[\gamma_{i1}]>
0$,
\item $H_{ij}[u_1,u_2]\geq \alpha_{ij}[u_i]$ for every $(u_1,u_2)\in K$,
\item $f_i(t,u_1,u_2)>P_i\, u_i$ for every $t\in [a,b]$ and $u_i \in (0,+\infty)$,
where
\end{itemize}
\begin{align*}
\dfrac{1}{P_i}:=\inf_{t \in [a_i,b_i] }\Bigl\{&\int_{a_{i}}^{b_{i}} k_{i}(t,s)g_i(s)\,ds+\Bigl(\dfrac{\gamma_{i1}(t)}{{D}_i}(1-\alpha_{i2}[\gamma_{i2}])+
\dfrac{\gamma_{i2}(t)}{{D}_i}\alpha_{i2}[\gamma_{i1}]\Bigr) \int_{a_{i}}^{b_{i}}\mathcal{K}_{i1}(s)g_{i}(s)\,ds\\
& +\Bigl(\dfrac{\gamma_{i1}(t)}{{D}_i}\alpha_{i1}[\gamma_{i2}]+
\dfrac{\gamma_{i2}(t)}{{D}_i}(1-\alpha_{i1}[\gamma_{i1}])\Bigr)\int_{a_i}^{b_i}
 \mathcal{K}_{i2}(s)g_i(s)\,ds \Bigr\}.
\end{align*}
Then there is no nontrivial solution of the system \eqref{syst} in $K$.
\end{thm}

\begin{proof}
Assume that there exists $(u,v)\in K$ such that
$(u,v)=T(u,v)$ and $(u,v)\neq (0,0)$. Let, for example, be
$\|u\|\neq 0$ with $\displaystyle\min_{t\in
[a_1,b_1]}u(t)=\theta>0$.  Then we obtain, for $t\in [a_1,b_1]$,
\begin{equation}\label{diss1nn}
u(t)= \sum_{j=1,2}\gamma_{1j}(t)H_{1j}[u,v]+F_{1}(u,v)(t)\,.
\end{equation}
Applying $\alpha _{1l}$, $l=1,2$, to both sides of
\eqref{diss1nn} gives
\begin{align*}
\alpha_{1l}[u]&=
\sum_{j=1,2}\alpha_{1l}\left[\gamma_{1j}\right]H_{1j}[u,v]+\alpha_{1l}\left[F_{1}(u,v)\right]
 \geq \sum_{j=1,2}\alpha_{1l}\left[\gamma_{1j}\right]\alpha_{1j}[u]+\alpha_{1l}\left[F_{1}(u,v)\right].
\end{align*}
In a similar way as in the proof of Theorem~\ref{nonexi1s}, we obtain,
for $t\in [a_1,b_1]$, 
\begin{align*}
 u(t)
>&\Bigl(\dfrac{\gamma_{11}(t)}{{D}_1}(1-\alpha_{12}[\gamma_{12}])+
\dfrac{\gamma_{12}(t)}{{D}_1}\alpha_{12}[\gamma_{11}]\Bigr)\int_{a_1}^{b_1} \mathcal{K}_{11}(s)g_1(s)P_1 u(s)\,ds\\
&+\Bigl(\dfrac{\gamma_{11}(t)}{{D}_1}\alpha_{11}[\gamma_{12}]+
\dfrac{\gamma_{12}(t)}{D_1}(1-\alpha_{11}[\gamma_{11}])\Bigr) \int_{a_1}^{b_1} \mathcal{K}_{121}(s)g_1(s)P_1u(s)\,ds\\
&+\int_{a_1}^{b_1} k_{1}(t,s)g_1(s)P_1u(s)\,ds.
\end{align*}
Taking the minimum over $[a_{1},b_{1}]$ gives
 $\theta> \theta$, a contradiction.
\end{proof}
\begin{thm} Assume that for $i,j=1,2$ there exist linear functionals $\alpha_{ij}[\cdot]:\tilde{K_{i}}\rightarrow [0,+\infty)$ given by
  $$
\alpha^{\rho_1,\rho_2}_{ij}[u]=\int_0^1 u(t)\,dA_{ij}(t)
$$ such that the assumptions in Theorem\,\ref{nonexi1s} are verified for an $i\in\{1,2\}$ and the assumptions in Theorem~\ref{noexi2 n} are verified for  the other index.
Then there is no nontrivial solution of the system~\eqref{syst} in
$K$.
\end{thm}
\begin{proof}
Assume, on the contrary, that there exists $(u,v)\in K$ such that
$(u,v)=T(u,v)$ and $(u,v)\neq (0,0)$. Let, for example, be
$\|u\|\neq 0$. Then the functions $\gamma_{11}$, $\gamma_{12}$,
$H_{11}$, $H_{12}$ and $f_1$ satisfy either the assumptions in
Theorem\,\ref{nonexi1s} or the assumptions in Theorem\,\ref{noexi2
n} and the proof follows as in the previous Theorems.
\end{proof}
\subsection{Radial solutions of systems of elliptic PDEs in annular domains}
The problem of the existence of positive \emph{radial} solutions
of elliptic equations subject to
non-homogeneous and nonlinear BCs on annular domains, has been investigated,
via different methods, by a number of authors, we refer the reader
to~\cite{dolo0, dolo, dolo3, dunn-wang, lee, waya} and references
therein.  Nonlocal BCs have been also studied  in the the context
of elliptic problems, we mention the
papers~\cite{bisamarADAN69,  gipp-nodea, skuba1, skuba2, wangy,
webb}.

Here we focus on the systems of BVPs
\begin{gather}
\begin{aligned}\label{ellbvp-secapp}
\Delta u + \tilde{g}_1(|x|) f_1(u,v)=0,\ |&x|\in [R_1,R_0], \\
\Delta v + \tilde{g}_2(|x|) f_2(u,v)=0,\ |&x|\in [R_1,R_0],\\
\frac{\partial u}{\partial r}\Bigr\rvert_{\partial
B_{R_0}}=\tilde{H}_{11}[u,v]\ \text{and}\
(u(R_1 x)-\beta_1  &u(R_\eta x))\Big|_{x\in\partial B_{1}}=H_{12}[u,v],\\
\frac{\partial v}{\partial r}\Bigr\rvert_{\partial
B_{R_0}}=\tilde{H}_{21}[u,v] \ \text{and}\
\bigl(v(R_1 x)-\tilde{\beta}_2 & \frac{\partial v}{\partial r}(R_\xi x)\bigr)\Big|_{x\in\partial B_{1}}=H_{22}[u,v],%
\end{aligned}
\end{gather}
where $x\in \mathbb{R}^n $, $\beta_1,\,\tilde{\beta}_2<0$,
$0<R_1<R_0<+\infty$, $R_\eta, R_\xi \in (R_1,R_0)$, that can be seen as a generalization of system (7.1) studied in \cite{gipp-nodea} (the range of $\tilde{\beta}_2$  corrects the one of $\alpha_2$ in \cite{gipp-nodea}).

Consider  in $\mathbb{R}^n$, $n\ge 2$, the equation
\begin{equation}\label{eqell}
\triangle w+ \tilde{g}(|x|)f(w) = 0, \  \text{for a.e.}\  |x|\in
[R_{1},R_{0}].
\end{equation}
In order to establish the existence of radial solutions $w=w(r)$,
$r=|x|$, we proceed as in~\cite{lan, lan-lin-na, lanwebb} and we
rewrite \eqref{eqell} in the form
\begin{equation}\label{eqinterm}
w''(r) + \dfrac{n-1}{r}w'(r) + \tilde{g}(r)f(w(r))= 0
\quad\text{a.e. on } [R_{1}, R_{0}].
\end{equation}
Set
$w(t)=w(r(t))$, where, for $t\in[0,1]$,
\begin{equation*}
r(t):=\begin{cases}
R_0^{1-t}R_1^{t},\,\,\,\,\,\,\,\,\,\,\,\,\,\,\,\,\,\, n=2,\\
({R_{0}^{-(n-2)}}+({R_{1}^{-(n-2)}}-{R_{0}^{-(n-2)}})t)^{-1/(n-2)},\ &n\geq 3.
\end{cases}
\end{equation*}
Take, for $t\in[0,1]$,
\begin{equation*}
\phi(t):=\begin{cases}
r^2(t) \log^2(R_0/R_1),\,\,\,\,\,\,\,\,\,\,\,\,\,\,\,\,\,\, n=2,\\
\Bigl(\frac{R_{1}^{-(n-2)}-R_{0}^{-(n-2)}}{n-2}\Bigr)^2 \Bigl(R_{0}^{-(n-2)}+(R_{1}^{-(n-2)}-R_{0}^{-(n-2)})t \Bigr)^{\frac{-2(n-1)}{n-2}},\ n\geq 3,
\end{cases}
\end{equation*}
then the equation ~\eqref{eqinterm}
becomes
\begin{equation*}
w''(t) + {\phi}(t) \tilde{g}(r(t)) f(w(t)) = 0, \  \text{a.e. on}\
[0,1].
\end{equation*}
 Set
$u(t)=u(r(t))$ and
$v(t)=v(r(t))$. Thus, to the system \eqref{ellbvp-secapp} we associate the system of ODEs
\begin{gather}
\begin{aligned}\label{1syst}
u''(t) + g_1(t) f_1(t,u(t),v(t)) = 0, \quad \text{a.e. on } [0,1], \\
v''(t) + g_2(t) f_2(t,u(t),v(t)) = 0, \quad \text{a.e. on } [0,1],%
\end{aligned}
\end{gather}
with the BCs
\begin{gather}
\begin{aligned}\label{1BC}
u'(0)+H_{11}[u,v]=0,\ u(1)={\beta}_1u({\eta})+H_{12}[u,v], \; 0 < {\eta}<1, \\
v'(0)+H_{21}[u,v]=0,\ v(1)=\beta_2 v'(\xi)+H_{22}[u,v], \; 0 < {\xi}<1,
\end{aligned}
\end{gather}
where
\begin{equation*}
g_i(t):={\phi}(t) \tilde{g}_i(r(t)),\,\,\beta_2=\begin{cases}\frac{\tilde{\beta}_2}{\log\left(R_1/R_0\right)\, R_\xi},&n=2,\\
-\frac{\tilde{\beta}_2 (n-2)}{R_\xi}\,\frac{\,R_1^{n-2}+\xi(R_0^{n-2}-R_1^{n-2})}{
R_0^{n-2}-R_1^{n-2}},\,& n\geq 3,
\end{cases}
\end{equation*}
\begin{equation*}
H_{11}[u,v]=\begin{cases}R_0\,\log\left(R_0/R_1\right)\tilde{H}_{11}[u,v],&n=2, \\
\displaystyle\frac{1}{n-2}\,
R_0\left(-1+\left(R_0/R_1\right)^{n-2}\right)\tilde{H}_{11}[u,v],&n\geq3,
\end{cases}\end{equation*}
\begin{equation*}
H_{21}[u,v]=\begin{cases}R_0\,\log\left(R_0/R_1\right)\tilde{H}_{21}[u,v],&n=2,\\
\displaystyle\frac{1}{n-2}\,
R_0\left(-1+\left(R_0/R_1\right)^{n-2}\right)\tilde{H}_{21}[u,v],&n\geq3,
\end{cases}\end{equation*}
and $\xi, \eta \in (0,1)$ are such that $r(\eta)=R_{\eta}$ and
$r(\xi)=R_{\xi}$.

Here we focus on the case $\beta_1<0$,
$0<\beta_2<1-\xi$, that leads to the case of solutions that are
\emph{positive} on some sub-intervals of $ [0,1]$ and are allowed
to change sign elsewhere.

We  study the existence of solutions of the
system~\eqref{1syst}-\eqref{1BC} by means of the  system
\begin{gather}
\begin{aligned}\label{3syst}
u(t)=&\Bigl(-t+\frac{1-\beta_1\eta}{1-\beta_1}\Bigr) H_{11}[u,v]+\frac{1}{1-\beta_1}H_{12}[u,v]+\int_{0}^{1}k_1(t,s)g_1(s)f_1(s,u(s),v(s))\,ds, \\
v(t)=&(-t+1-\beta_2)H_{21}[u,v]+H_{22}[u,v]+\int_{0}^{1}k_2(t,s)g_2(s)f_2(s,u(s),v(s))\,ds,%
\end{aligned}
\end{gather}
where the Green's functions are given by
\begin{equation*} 
k_1(t,s)=\dfrac{1}{1-\beta_1}(1-s)-\begin{cases}
\dfrac{\beta_1}{1-\beta_1}(\eta -s), &  s \le \eta,\\ \quad 0,&
s>\eta,
\end{cases}
 - \begin{cases} t-s, &s\le t, \\ \quad 0,&s>t,
\end{cases}
\end{equation*}
and
\begin{equation*}
k_2(t,s)=(1-s)-\begin{cases} \beta_2, &  s \le \xi,\\ \quad 0,& s>\xi,
\end{cases}
 - \begin{cases} t-s, &s\le t, \\ \quad 0,&s>t.
\end{cases}
\end{equation*}
The Green's function $k_1$ has been studied in \cite{gijwjiea},
where it was shown that we may take $$\Phi_1(s)= 1-s,$$ arbitrary
$[a_1,b_1]\subset[0,\eta]$ and
$\tilde{c}_1={(1-\eta)}/{(1-\beta_1)}$.

Regarding $k_2$, this has been studied in \cite{giems}; we may take
$$
\Phi_2(s)=1-s,
$$
arbitrary $[a_2,b_2]\subset [0,\xi]$ and $ \tilde{c}_2=1-\beta_2-\xi$.

By direct calculation we obtain $$||\gamma_{11}||=\frac{1-\beta_1\eta}{1-\beta_1},\,\,
c_{11}=\frac{1-\eta}{1-\beta_1\eta},\,\,
||\gamma_{12}||=\frac{1}{1-\beta_1},\,\, c_{12}=1,$$$$
||\gamma_{21}||=1-\beta_2,\,\,
c_{21}=\frac{1-\beta_2-\xi}{1-\beta_2},\, \,||\gamma_{22}||=1,\,\,
c_{22}=1.$$ Thus we work in the cone
\begin{equation*}
K=\{u\in C[0,1]:\,\,\min_{t\in [0,\eta]}u(t)\geq {c_{1}}\|
u\|\}\times \{v\in C[0,1]:\,\,\min_{t\in [0,\xi]}v(t)\geq {c_{2}}\|
v\|\}
\end{equation*}
where
 $c_{1}=\min \{\tilde{c}_{1},c_{11},c_{12}\}=\frac{1-\eta}{1-\beta_1}$ and $c_{2}=\min \{\tilde{c}_{2},c_{21},c_{22}\}=1-\beta_2-\xi$.

The results of the previous Subsections can be applied to the system~\eqref{3syst}, yielding results for the system \eqref{ellbvp-secapp}, we refer to \cite{lan-lin-na, lanwebb} for the results that may be stated.

We conclude by showing in the following example that all the
constants that occur in Theorem~\ref{mult-sys} can be computed.
\begin{ex}
Consider  in $\mathbb{R}^2$  the system of BVPs
\begin{gather}
\begin{aligned}\label{ellbvpex}
\Delta u + \frac{1}{4}(|u|^3+|v|^3+1)=&0,\ |x|\in [1,e], \\
\Delta v + \frac{1}{8}(|u|^{\frac{1}{2}}+v^2+1)=&0,\ |x|\in [1,e],\\
\frac{\partial u}{\partial r}\bigr\rvert_{\partial B_{e}}=0,\quad
(u(x)+  u(\sqrt{e} x))|_{x\in\partial B_{1}}=&\bigl(\frac{3}{40} u^2(\sqrt[6]{e^5} x)+\sqrt{\frac{3}{40}} v^3(\sqrt[5]{e^4} x)\bigr)|_{x\in\partial B_{1}},\\
\frac{\partial v}{\partial r}\bigr\rvert_{\partial B_{e}}=0, \quad
\bigl(v(x)+\frac{\sqrt[4]{e^3}}{4} & \frac{\partial v}{\partial r}(\sqrt[4]{e^3} x)\bigr)|_{x\in\partial B_{1}}=0.%
\end{aligned}
\end{gather}
The system~\eqref{ellbvpex} can be seen as a perturbation of the system (7.5) in~\cite{gipp-nodea} and also corrects the misprints in the BCs therein.
To the system~\eqref{ellbvpex} we associate the system of second order ODEs
\begin{align*}
&u''(t)+ \frac{1}{4} e^{2(1-t)} (|u(t)|^3+|v(t)|^3+1)=0,\ t\in [0,1], \\
&v''(t) +\frac{1}{8}  e^{2(1-t)} (|u(t)|^{\frac{1}{2}}+v^2(t)+1)=0,\ t\in [0,1],\\
&u'(0)=0,\,\,\,\,\,\; u(1/2)+u(1)=\dfrac{3}{40} u^2(1/6)+\sqrt{\dfrac{3}{40}} v^3(1/5),\\
&v'(0)=0,\;\,\,\,\,\, v'(1/4)=4v(1).%
\end{align*}

Since in  $[0,1]\times  [0,1]$ the kernel $k_1$ is not positive for $$\dfrac{1-\beta_1 \eta}{1-\beta_1}\leq t\leq1\ \text{and}\ 0\leq s \leq \dfrac{1-\beta_1}{-\beta_1}t+\dfrac{1}{\beta_1},$$
we have
\begin{align*}
 \frac{1}{m_{1}}=\sup_{t\in [0,1]}\int_{0}^{1}\vert k_{1}(t,s)\vert g_{1}(s)\,ds
 &=\max \Biggl\{
 \sup_{t\in [0,3/4]}
 \Bigl\{\frac {e^2}{2}\Bigl(-\frac{e^{-2t}}{2}-t+\frac{3+e^{-1}+e^{-2}}{4}\Bigr) \Bigr\} ,\\
 & \sup_{t\in [3/4,1]}
 \Bigl\{\frac {e^2}{8}\Bigl(2 e^2 e^{-4t}-2e^{-2t}+4t-3-e^{-1}\Bigr) +\frac{1}{8}\Bigr\}\Biggr\}.
 \end{align*}{}
Note that the kernel $k_2$ in  $[0,1]\times  [0,1$] is not positive for
$$1-\beta_2\leq t\leq1\ \text{and}\ 0\leq s \leq \xi;$$
we have
\begin{multline*}
 \frac{1}{m_{2}}=\sup_{t\in [0,1]}\int_{0}^{1}\vert k_{2}(t,s)\vert g_{2}(s)\,ds
 =\max \Biggl\{
 \sup_{t\in [0,3/4]}
 \Bigl\{\frac {e^2}{2}\Bigl(-\frac{e^{-2t}}{2}-t+\frac{3+e^{-\frac{1}{2}}}{4}\Bigr) +\frac{1}{4}\Bigr\} ,\\
  \sup_{t\in [3/4,1]}   \Bigl\{\frac {e^2}{2}\Bigr(-\frac{e^{-2t}}{2}-t(\frac{2}{\sqrt{e}}-1)+\frac{-3+7e^{-\frac{1}{2}}}{4}\Bigr) +\frac{1}{4}\Bigr\}\Biggr\} .
 \end{multline*}{}
We fix $[a_1,b_1]=[a_2,b_2]=[0,1/4]$, obtaining
$$
 \frac{1}{M_{1}} = \inf_{t\in
[0,1/4]}\int_{0}^{1/4} k_1(t,s) g_1(s)\,ds=\inf_{t\in [0,1/4]}\Bigl\{\frac {e^2}{8}\Bigl(-2e^{-2t}-4t+3\Bigr) \Bigr\} ,
$$
and
$$
 \frac{1}{M_{2}} = \inf_{t\in
[0,1/4]}\int_{0}^{1/4} k_2(t,s) g_2(s)\,ds=\inf_{t\in [0,1/4]}\Bigl\{ \frac {e^2}{2}\Bigl(-\frac{e^{-2t}}{2}-t+\frac{3+e^{-\frac{1}{2}}}{4}\Bigr) +\frac{1}{4} \Bigr\}.
$$
By direct computation, we get
$$
c_1=\frac{1}{4},\, m_1=0.72,\, M_1=2.16,\, c_2=\frac{1}{4},\, m_2=0.577,\, M_2=1.376.
$$
With the choice of
$$
\rho_1=\rho_2=1/12,\,\,\alpha^{\rho_1,\rho_2}_{121}[u]=\alpha^{\rho_1,\rho_2}_{122}[v]=0,
$$
$$
r_1=1,\,\, r_2=1/2,\,\, \alpha^{r_1,r_2}_{121}[u]=0.3\,
u(1/6),\,\, \alpha^{r_1,r_2}_{122}[v]=0.3\,
v(1/5),
$$
$$
s_1=5,\,\,s_2=11, \alpha^{s_1,s_2}_{121}[u]=1.5\,
u(1/6),\,\, \alpha^{s_1,s_2}_{122}[v]=1.5\, v(1/5),
$$
we obtain
$$
\alpha^{r_1,r_2}_{121}[\gamma_{12}]=\alpha^{r_1,r_2}_{122}[\gamma_{12}]=0.3\,\cdot \frac{1}{2}<1,\,\,\,\\
\,\,\,\,\alpha^{s_1,s_2}_{121}[\gamma_{12}]=\alpha^{s_1,s_2}_{122}[\gamma_{12}]=1.5\,\cdot \frac{1}{2}<1,
$$
\begin{align*}
H_{12}[u,v]=&\dfrac{3}{40} u^2(1/6)+\sqrt{\dfrac{3}{40}}
v^3(1/5)\leq
\alpha^{r_1,r_2}_{121}[u]+\alpha^{r_1,r_2}_{122}[v],\,\,\,
(u,v)\in \partial K_{r_1,r_2}\, \\
H_{12}[u,v]=&\dfrac{3}{40} u^2(1/6)+\sqrt{\dfrac{3}{40}}
v^3(1/5)\geq
\alpha^{s_1,s_2}_{121}[u]+\alpha^{s_1,s_2}_{122}[v],\,\,\,
(u,v)\in \partial V_{s_1,s_2}\,
\end{align*}
\begin{align*}
\inf & \Bigl\{ f_2(u,v):\; (u,v)\in [-4\rho_1,4\rho_1]\times[0,4\rho_2]
\Bigr\}= f_2(0,0)=0.125>1.376\cdot\frac{1}{12}, \\
\sup &\Bigl\{ f_1(u,v):\; (u,v)\in [-r_1, r_1]\times[-r_2,
r_2]\Bigr\}=f_1\left(1,\frac{1}{2}\right)=0.531
< 0.579  , \\
\sup &\Bigl\{ f_2(u,v):\; (u,v)\in [-r_1, r_1]\times[-r_2,
r_2]\Bigr\}=f_2\left(1,\frac{1}{2}\right)=0.281
<0.577\cdot\frac{1}{2}, \\
\inf &\Bigl\{ f_1(u,v):\; (u,v)\in
[s_1,4s_1]\times[-4s_2,4s_2]\Bigr\}=f_1(5,0)=31.5
>0.416  \cdot 5, \\
\inf &\Bigl\{ f_2(u,v):\; (u,v)\in [-4s_1,4s_1]\times[s_2,
4s_2]\Bigr\}=f_2(0,11)=15.25 >1.376 \cdot 11.
\end{align*}
Thus the conditions $(\mathrm{I}^{0}_{\rho_{1},\rho_2})^{\diamond}$, $(\mathrm{I}%
^{1}_{r_1,r_{2}})$ and $(\mathrm{I}^{0}_{s_1,s_{2}})$ are satisfied; therefore
the system~\eqref{ellbvpex} has at least two nontrivial solutions.
\end{ex}

\section*{Acknowledgments}
G. Infante and P. Pietramala were partially supported by G.N.A.M.P.A. - INdAM (Italy). F. Cianciaruso is a member of  G.N.A.M.P.A.


\begin{thebibliography}{00}

\bibitem{amp} E. Alves, T. F. Ma and M. L. Pelicer,
 Monotone positive solutions for a fourth order equation with nonlinear boundary conditions,
\textit{Nonlinear Anal.}, \textbf{71} (2009), 3834--3841.

\bibitem{Amann-rev} H. Amann,
Fixed point equations and nonlinear eigenvalue problems in ordered Banach spaces,
\textit{SIAM. Rev.}, \textbf{18} (1976), 620--709.

\bibitem{da-jh} D. R. Anderson and J. Hoffacker,
 Existence of solutions for a cantilever beam problem,
\textit{J. Math. Anal. Appl.}, \textbf{323} (2006), 958--973.

\bibitem{bisamarADAN69} A. V. Bitsadze and  A. A. Samarski\u\i,
Some elementary generalizations of linear elliptic boundary value problems (Russian),
 {\it Anal. Dokl. Akad. Nauk SSSR}, {\bf 185} (1969), 739--740.

\bibitem{BOKR}  D.M. Bo\v{s}kovi\'{c} and M. Krsti\'{c},
 Backstepping control of chemical tubular reactors,
 \textit{Computers and Chemical Engineering}, \textbf{261} (2002), 1077--1085.

\bibitem{Cabada1} A. Cabada,
An overview of the lower and upper solutions method with nonlinear
boundary value conditions, \textit{Bound. Value Probl.},
\textbf{2011} (2011), Art. ID 893753, 18 pp.

\bibitem{ac-gi-at-bvp}
A. Cabada, G. Infante and F. A. F. Tojo, Nontrivial solutions of perturbed Hammerstein integral equations with reflections, \textit{Bound. Value Probl.}, \textbf{2013:86} (2013), 22 pp.

\bibitem{ac-gi-at-fpt}
A. Cabada, G. Infante and F. A. F. Tojo, Nonlinear perturbed integral equations related to nonlocal boundary value problems, \textit{Fixed Point Theory}, to appear.

\bibitem{cate} A. Cabada and S. Tersian,
Multiplicity of solutions of a two point boundary value problem for a fourth-order equation,
\textit{Appl. Math. Comput.}, \textbf{219} (2013), 5261--5267.

\bibitem{chzh2}  X. Cheng and C. Zhong,
 Existence of positive solutions for a second-order ordinary differential system,
 \textit{J. Math. Anal Appl.}, \textbf{312} (2005), 14--23.

\bibitem{CLR} C. Chicone, Y. Latushkin and D. G. Retzloff,
Chemical reactor dynamics: stability of steady states,
\textit{Math. Methods Appl. Sci.}, \textbf{19} (1996),  381--400.

\bibitem{Conti} R. Conti,
Recent trends in the theory of boundary value problems for ordinary differential equations,
\textit{Boll. Un. Mat. Ital.}, \textbf{22} (1967), 135--178.

 \bibitem{dolo0} J. M. do {\'O}, S. Lorca and P. Ubilla,
Local superlinearity for elliptic systems involving parameters,
\textit{J. Differential Equations}, \textbf{211} (2005), 1--19.

\bibitem{dolo} J. M. do {\'O}, S. Lorca and P. Ubilla,
Three positive solutions for a class of elliptic systems in annular domains,
\textit{Proc. Edinb. Math. Soc.}, \textbf{48} (2005),  365--373.

\bibitem{dolo3} J. M. do {\'O}, J. S{\'a}nchez, S. Lorca and P. Ubilla,
Positive solutions for a class of multiparameter ordinary elliptic systems,
\textit{J. Math. Anal. Appl.}, \textbf{332} (2007), 1249--1266.

\bibitem{dunn-wang} D. R. Dunninger and H. Wang, Multiplicity of
positive solutions for a nonlinear differential equation with
nonlinear boundary conditions, \textit{Ann. Polon. Math.},
\textbf{69} (1998), 155--165.

\bibitem{fama}
 H. Fan and R. Ma, Loss of positivity in a nonlinear second order ordinary differential equations, \textit{Nonlinear Anal.}, \textbf{71} (2009), 437--444.

\bibitem{feng}
W. Feng, G. Zhang and Y. Chai, Existence of positive solutions for second order differential equations arising from chemical reactor theory, \textit{Discrete Contin. Dyn. Syst. 2007}, Dynamical Systems and Differential Equations. Proceedings of the 6th AIMS International Conference, suppl., 373--381.

\bibitem{df-gi-do} D. Franco, G. Infante and D. O'Regan,
Nontrivial solutions in abstract cones for Hammerstein integral
systems, \textit{Dyn. Contin. Discrete Impuls. Syst. Ser. A Math.
Anal.}, {\bf 14} (2007), 837--850.

\bibitem{df-gi-jp-prse}
D. Franco, G. Infante and J. Per\'an, A new criterion for the existence of multiple solutions in cones, \textit{Proc. Roy. Soc. Edinburgh Sect. A}, {\bf 142} (2012), 1043--1050.

\bibitem{dfdorjp} D. Franco, D. O'Regan and J. Per\'an,
Fourth-order problems with nonlinear boundary conditions,
\textit{J. Comput. Appl. Math.}, \textbf{174} (2005), 315--327.

\bibitem{Goodrich2} C. S. Goodrich,
Nonlocal systems of BVPs with asymptotically sublinear boundary
conditions, \textit{Appl. Anal. Discrete Math.}, \textbf{6}
(2012), 174--193.

\bibitem{Goodrich1} C. S. Goodrich,
Nonlocal systems of BVPs with asymptotically superlinear boundary
conditions, \textit{Comment. Math. Univ. Carolin.}, \textbf{53}
(2012), 79--97.

\bibitem{Goodrich3}
C. S. Goodrich, On nonlocal BVPs with nonlinear boundary
conditions with asymptotically sublinear or superlinear growth,
\textit{Math. Nachr.}, \textbf{285} (2012), 1404--1421.

\bibitem{Goodrich5}
C. S. Goodrich, On nonlinear boundary conditions satisfying
certain asymptotic behavior, \textit{Nonlinear Anal.}, \textbf{76}
(2013), 58--67.

\bibitem{good}
C. S. Goodrich, On a nonlocal BVP with nonlinear boundary
conditions, \textit{Results Math.}, \textbf{63} (2013),
1351--1364.

\bibitem{Goodrich6}
C. S. Goodrich, A note on semipositone boundary value problems
with nonlocal, nonlinear boundary conditions, \textit{Arch. Math.
(Basel)}, \textbf{103} (2014),  177--187.

\bibitem{good1}
C. S. Goodrich,  Coupled systems of boundary value problems with
nonlocal boundary conditions, \textit{Appl. Math. Lett.},
\textbf{41} (2015), 17--22.

\bibitem{Goodrich7}
C. S. Goodrich, Semipositone boundary value problems with
nonlocal, nonlinear boundary conditions, \textit{Adv. Differential
Equations}, \textbf{20} (2015), 117--142.

\bibitem{Guidotti} P. Guidotti and S. Merino,
 Gradual loss of positivity and hidden invariant cones in a scalar heat equation,
\textit{Differential Integral Equations}, \textbf{13} (2000),
1551-1568.

\bibitem{guolak} D. Guo and V. Lakshmikantham,
\textit{Nonlinear Problems in Abstract Cones},  Academic Press,
Boston, (1988).

\bibitem{HOTI}
A. Horvat-Marc and C. {\c{T}}ical{\u{a}}, Localization of
solutions for a problem arising in the theory of adiabatic tubular
chemical reactors, \textit{Carpathian J. Math.}, \textbf{20}
(2004), 187--192.

\bibitem{giems} G. Infante, Eigenvalues of some non-local boundary-value problems,
\textit{Proc. Edinb. Math. Soc.}, \textbf{46} (2003), 75--86.

\bibitem{gi-caa} G. Infante,
 Nonlocal boundary value problems with two nonlinear boundary conditions,
\textit{Commun. Appl. Anal.}, \textbf{12} (2008), 279--288.

\bibitem{gifmpp-cnsns} G. Infante, F. M. Minh\'os and P. Pietramala,
Non-negative solutions of systems of ODEs with coupled boundary
conditions, \textit{ Commun. Nonlinear Sci. Numer. Simul.},
\textbf{17} (2012), 4952--4960.

\bibitem{cantilever} G. Infante and P. Pietramala,
A cantilever equation with nonlinear boundary conditions,
\textit{Electron. J. Qual. Theory Differ. Equ., Spec. Ed. I},
\textbf{15} (2009), 1--14.

 \bibitem{gipp-ns} G. Infante and P. Pietramala,
Eigenvalues and non-negative solutions of a system with nonlocal
BCs, \textit{Nonlinear Stud.}, \textbf{16} (2009), 187--196.

\bibitem{gipp-nonlin} G. Infante and P. Pietramala,
 Existence and multiplicity of non-negative solutions for systems of perturbed Hammerstein integral equations,
 \textit{Nonlinear Anal.}, \textbf{71} (2009), 1301--1310.

\bibitem{gipp cmuc} G. Infante and P. Pietramala, Perturbed Hammerstein integral inclusions with solutions that change sign, \textit{Comment. Math. Univ. Carolin.}, \textbf{50} (2009), 591--605.

 \bibitem{gipp-mmas} G. Infante and P. Pietramala,
 Multiple nonnegative solutions of systems with coupled nonlinear boundary conditions,
 \textit{Math. Methods Appl. Sci.}, \textbf{37} (2014), 2080--2090.

\bibitem{gipp-nodea} G. Infante and P. Pietramala,
 Nonzero radial solutions for a class of elliptic systems with nonlocal BCs on annular domains,
 \textit{NoDEA Nonlinear
Differential Equations Appl.}, \textbf{22} (2015), 979--1003.

 \bibitem{gippreactor}  G. Infante, P. Pietramala and M. Tenuta,
Existence and localization of positive solutions for a nonlocal
BVP arising in chemical reactor theory, \textit{Commun. Nonlinear
Sci. Numer. Simul.}, \textbf{19} (2014), 2245--2251.

\bibitem{gi-pp-ft}
 G. Infante, P. Pietramala and F. A. F. Tojo,
 Nontrivial solutions of local and nonlocal Neumann boundary value problems, \textit{Proc. Roy. Soc. Edinburgh Sect. A}, to appear.

\bibitem{gijwjiea} G. Infante and J. R. L. Webb,
Three point boundary value problems with solutions that change
sign, \textit{J. Integral Equations Appl.}, \textbf{15} (2003),
37--57.

\bibitem{gijwnodea} G. Infante and J. R. L. Webb, Loss of positivity in
a nonlinear scalar heat equation, \textit{NoDEA Nonlinear
Differential Equations Appl.}, \textbf{13} (2006), 249-261.

\bibitem{gijwems} G. Infante and J. R. L. Webb,
  Nonlinear nonlocal boundary value problems and perturbed Hammerstein integral equations,
  \textit{Proc. Edinb. Math. Soc.}, \textbf{49} (2006), 637--656.

\bibitem{kamkee}
G. Kalna and S. McKee, The thermostat problem,  \textit{TEMA Tend.
Mat. Apl. Comput.}, \textbf{3 }(2002), 15--29.

\bibitem{kamkee2}
G. Kalna and S. McKee, The thermostat problem with a nonlocal
nonlinear boundary condition, \textit{IMA J. Appl. Math.},
\textbf{69} (2004), 437--462.

\bibitem{kang-wei} P. Kang and Z. Wei,
Three positive solutions of singular nonlocal boundary value
problems for systems of nonlinear second-order ordinary
differential equations, \textit{Nonlinear Anal.}, \textbf{70}
(2009), 444--451.

\bibitem{kejde}
G. L. Karakostas, Existence of solutions for an $n$-dimensional
operator equation and applications to BVPs, \textit{Electron. J.
Differential Equations}, \textbf{2014}, 17 pp.

\bibitem{kttmna} G. L. Karakostas and P. Ch. Tsamatos, Existence of multiple
positive solutions for a nonlocal boundary value problem,
\textit{Topol. Methods Nonlinear Anal.}, \textbf{19} (2002),
109--121.

\bibitem{ktejde}
G. L. Karakostas and P. Ch. Tsamatos, Multiple positive solutions
of some Fredholm integral equations arisen from nonlocal
boundary-value problems, \textit{Electron. J. Differential
Equations}, \textbf{2002}, 17 pp.

\bibitem{kapala}
I. Karatsompanis and P. K. Palamides, Polynomial approximation to
a non-local boundary value problem, \textit{Comput. Math. Appl.},
\textbf{60} (2010),  3058--3071.

\bibitem{kongwang} L. Kong and J. Wang, Multiple positive solutions
for the one-dimensional
 $p$-Laplacian, {\it Nonlinear Anal.}, {\bf 42} (2000), 1327--1333.

\bibitem{krzab} M. A. Krasnosel'ski\u\i{} and P. P. Zabre\u{\i}ko,
{\it Geometrical methods of nonlinear analysis}, Springer-Verlag,
Berlin, (1984).

\bibitem{kljdeds} K. Q. Lan,
Multiple positive solutions of Hammerstein integral equations with
singularities, \textit{Diff. Eqns and Dynam. Syst.}, \textbf{8}
(2000), 175--195.

\bibitem{lan} K. Q. Lan,
{Multiple positive solutions of semilinear differential equations
with singularities,} \textit{J. London Math. Soc.}, \textbf{63}
(2001), 690--704.

\bibitem{lan-lin-na} K. Q. Lan and W. Lin,
Positive solutions of systems of singular Hammerstein integral
equations with applications to semilinear elliptic equations in
annuli, \textit{Nonlinear Anal.}, \textbf{74} (2011), 7184--7197.

\bibitem{lanwebb} K. Q. Lan and J. R. L. Webb,
Positive solutions of semilinear differential equations with
singularities, \textit{J. Differential Equations}, \textbf{148}
(1998), 407--421.

\bibitem{lee} Y. H. Lee,
Multiplicity of positive radial solutions for multiparameter
semilinear elliptic systems on an annulus, \textit{J. Differential
Equations}, \textbf{174} (2001), 420--441.

\bibitem{li}
Y. Li, Existence of positive solutions for the cantilever beam
equations with fully nonlinear terms, \textit{Nonlinear Anal. Real
World Appl.}, \textbf{27} (2016), 221--237.

\bibitem{lizhai}
S. Li and C. Zhai, New existence and uniqueness results for an
elastic beam equation with nonlinear boundary conditions,
\textit{Bound. Value Probl.}, \textbf{2015:104} (2015), 12 pp.

\bibitem {rma}
R. Ma, A survey on nonlocal boundary value problems, \textit{Appl.
Math. E-Notes}, \textbf{7} (2001), 257--279.

\bibitem{mada}
T. F. Ma and J. da Silva, Iterative solutions for a beam equation
with nonlinear boundary conditions of third order, \textit{Appl.
Math. Comput.}, \textbf{159} (2004),  11--18.

\bibitem{madbouly}
 N. M. Madbouly, D. F. McGhee and G. F. Roach, Adomian's method for Hammerstein integral equations arising from chemical reactor theory,
\textit{Appl. Math. Comput.}, \textbf{117} (2001), 241-249.

\bibitem{Maam}
 L. Markus and N. R. Amundson,
 Nonlinear boundary-value problems arising in chemical reactor theory, \textit{J. Differential Equations}, \textbf{4} (1968) 102--113.

\bibitem{nipime}
J. J. Nieto and J. Pimentel, Positive solutions of a fractional
thermostat model, \textit{Bound. Value Probl.}, \textbf{2013:5}
(2013), 11 pp.

\bibitem{sotiris}
S. K. Ntouyas, Nonlocal initial and boundary value problems: a
survey, \textit{Handbook of differential equations: ordinary
differential equations. Vol. II}, Elsevier B. V., Amsterdam,
(2005), 461--557.

\bibitem{palamides}
P. Palamides, G. Infante and P. Pietramala, Nontrivial solutions
of a nonlinear heat flow problem via Sperner's Lemma,
\textit{Appl. Math. Lett.},  \textbf{22} (2009), 1444--1450.

\bibitem{Picone} M. Picone, Su un problema al contorno nelle equazioni differenziali lineari ordinarie del secondo ordine,
\textit{Ann. Scuola Norm. Sup. Pisa Cl. Sci.}, \textbf{10} (1908), 1--95.

\bibitem{paola}  P. Pietramala,
A note on a beam equation with nonlinear boundary conditions,
\textit{Bound. Value Probl.}, \textbf{2011}  (2011), Art. ID
376782, 14 pp.

\bibitem{precup1} R. Precup,
Componentwise compression-expansion conditions for systems of
nonlinear operator equations and applications,
\textit{Mathematical models in engineering, biology and medicine}, AIP Conf. Proc., \textbf{1124}, Amer. Inst. Phys., Melville, NY,
(2009), 284--293.

\bibitem{precup2} R. Precup,
Existence, localization and multiplicity results for positive radial solutions of semilinear elliptic systems,
\textit{J. Math. Anal. Appl.}, \textbf{352} (2009), 48--56.

\bibitem{RCH}
D. G. Retzloff, C. Chicone and G. H. Hsu, Multiple solutions of a
nonlinear boundary value problem with application to chemical
reactor dynamics, \textit{J. Math. Anal. Appl.}, \textbf{185}
(1994), 501--519.

 \bibitem{Saad}
A. Saadatmandi, M. Razzaghi and M. Dehghan, Sinc-Galerkin solution
for nonlinear two- point boundary value problems with applications
to chemical reactor theory, \textit{Math. Comput. Modelling},
\textbf{42} (2005), 1237--1244.

\bibitem{skuba1} A. L. Skubachevski{\u\i}, Nonclassical boundary value problems. I,
\textit{J. Math. Sci. (N. Y.)}, \textbf{155} (2008), 199--334.

\bibitem{skuba2} A. L. Skubachevski{\u\i}, Nonclassical boundary value problems. II,
\textit{J. Math. Sci. (N. Y.)}, \textbf{166} (2010), 377--561.

\bibitem{song}
 Y. Song, A nonlinear boundary value problem for fourth-order elastic beam equations, \textit{Bound. Value Probl.}, \textbf{2014:191} (2014), 11 pp.

\bibitem{Stik} A. \v{S}tikonas, A survey on stationary problems, Green's functions and spectrum of Sturm-Liouville problem
with nonlocal boundary conditions, \textit{Nonlinear Anal. Model.
Control}, \textbf{19} (2014), 301--334.

\bibitem{wang} J. Wang, The existence of positive solutions for
the one-dimensional $p$-Laplacian,
\textit{Proc. Amer. Math. Soc.}, {\bf 125} (1997), 2275--2283.

\bibitem{wangy} Y. Wang,
Solutions to nonlinear elliptic equations with a nonlocal boundary
condition, \textit{Electron. J. Differential Equations},
\textbf{05}  (2002), 16 pp.

\bibitem{waya}
C. Wang and J. Yang,  The existence of positive solutions to an
elliptic system with nonlinear boundary conditions, \textit{Bound.
Value Probl.}, \textbf{2013:159} (2013), 17 pp.

\bibitem{webb} J. R. L. Webb,
 Positive solutions of some three point boundary value problems via fixed point index theory,
 \textit{Nonlinear Anal.}, \textbf{47} (2001), 4319--4332.

\bibitem{jwpomona}
J. R. L. Webb, Multiple positive solutions of some nonlinear heat
flow problems, \textit{Discrete Contin. Dyn. Syst. (Suppl.)},
(2005), 895--903.

\bibitem{jwwcna04} J. R. L. Webb,
Optimal constants in a nonlocal boundary value problem,
 \textit{Nonlinear Anal.}, \textbf{63} (2005), 672--685.

\bibitem{webb-therm}
J. R. L. Webb, Existence of positive solutions for a thermostat
model, \textit{Nonlinear Anal. Real World Appl.}, \textbf{13}
(2012), 923--938.

\bibitem{jw-gi-jlms} J. R. L. Webb and G. Infante,
Positive solutions of nonlocal boundary value problems: a unified
approach, \textit{J. London Math. Soc.}, \textbf{74} (2006),
673--693.

\bibitem{Whyburn}
W. M. Whyburn, Differential equations with general boundary
conditions, \textit{Bull. Amer. Math. Soc.}, \textbf{48} (1942),
692--704.

\bibitem{ya1} Z. Yang,
Positive solutions to a system of second-order nonlocal boundary
value problems, \textit{Nonlinear Anal.}, \textbf{62} (2005),
1251--1265.

\bibitem{ya2}
Z. Yang, Positive solutions of a second-order integral boundary
value problem, \textit{J. Math. Anal. Appl.}, \textbf{321} (2006),
751--765.

\bibitem{yao-cant}
Q. Yao, Monotonically iterative method of nonlinear cantilever beam equations,
\textit{Appl. Math. Comput.}, \textbf{205} (2008), 432--437.

\end{thebibliography}
\end{document}